\documentclass[11pt,reqno,letter]{amsart}
\usepackage{fullpage}

\usepackage{amssymb,amsmath,amsthm}
\usepackage{comment,enumitem,url}
\usepackage{xcolor}
\usepackage{enumitem}
\usepackage{bm}
\usepackage{float}
\usepackage{bbm}
\usepackage{booktabs}
\usepackage{makecell}

\setlist[enumerate]{leftmargin=*,noitemsep, topsep=0pt}
\setlist[itemize]{leftmargin=*,noitemsep, topsep=0pt}

\renewcommand{\pmod}[1]{\ \, \left( \mathrm{mod} \, #1 \right)}
\newcommand{\Pmod}[1]{\ \, ( \mathrm{mod} \, #1 )}

\newcommand{\GG}{\mathcal{G}}

\newcommand{\N}{\mathbb{N}}

\newcommand{\g}{\gamma}

\newcommand{\ve}{\varepsilon}

\newcommand{\coef}[1]{\operatorname{coeff}_{\left[#1\right]}}

\usepackage{mathtools}

\renewcommand{\pmod}[1]{\ \left( \mathrm{mod} \, #1 \right)}
\renewcommand{\Pmod}[1]{\ ( \mathrm{mod} \, #1 )}

\newcommand{\CA}{\mathcal{A}}

\newcommand{\CI}{\mathcal{I}}

\newcommand{\IR}{{\mathbb R}}
\newcommand{\IC}{{\mathbb C}}
\newcommand{\IZ}{{\mathbb Z}}
\newcommand{\IN}{{\mathbb N}}

\theoremstyle{plain}
\newtheorem{thm}{Theorem}[section]
\newtheorem{cor}[thm]{Corollary}
\newtheorem{lem}[thm]{Lemma}
\newtheorem{prop}[thm]{Proposition}

\newtheorem*{rem}{Remark}

\theoremstyle{definition}

\numberwithin{table}{section}

\newcommand{\mat}[1]{\left( \begin{matrix} #1 \end{matrix} \right)}

\def\lp{\left(}
\def\rp{\right)}
\def\lsp{(}
\def\rsp{)}
\def\lb{\left[}
\def\rb{\right]}
\def\lsb{[}
\def\rsb{]}

\newcommand{\andd}{\quad \mbox{ and } \quad}
\newcommand{\where}{\quad \mbox{ where }}

\usepackage{enumitem}
\setlist[itemize]{noitemsep, topsep=0pt}
\setlist[enumerate]{noitemsep, topsep=0pt}

\def\a{\alpha}
\def\b{\beta}

\def\k{\bm{k}}
\def\l{\lambda}

\def\w{\omega}

\def\th{\theta}

\def\ve{\varepsilon}

\def\s{\sigma}
\def\g{\gamma}

\def\GG{\Gamma}
\def\DD{\Delta}
\def\LL{\Lambda}

\allowdisplaybreaks

\makeatletter
\newcommand{\vast}{\bBigg@{3}}
\newcommand{\Vast}{\bBigg@{5}}
\makeatother

\newcommand{\Ssm}{{\scriptscriptstyle \Sigma}}

\title{Statistics for Random Representations of Lie algebras}

\author{Walter Bridges}
\address{University of North Texas, Department of Mathematics, Denton, TX, USA}
\email{Walter.Bridges@unt.edu}

\author{Kathrin Bringmann}
\address{University of Cologne, Department of Mathematics and Computer Science, Cologne, Germany}
\email{kbringma@uni-koeln.de}

\author{Caner Nazaroglu}
\address{University of Cologne, Department of Mathematics and Computer Science, Cologne, Germany}
\email{cnazarog@uni-koeln.de}

\makeatletter
\@namedef{subjclassname@2020}{%
\textup{2020} Mathematics Subject Classification}
\makeatother


\begin{document}

\begin{abstract}
In this paper we investigate how a typical, large-dimensional representation looks for a complex Lie algebra. In particular, we study the family $\mathfrak{sl}_{r+1}(\mathbb{C})$ of Lie algebras for $r \geq 2$ and derive asymptotic probability distributions for the multiplicity of small irreducible representations, as well as the largest dimension, the largest height, and the total number of irreducible representations appearing in the decomposition of a representation sampled uniformly from all representations with the same dimension.
This provides a natural generalization to the similar statistical studies of integer partitions, which forms the case $r=1$ of our considerations and where one has
a rich toolkit ranging from combinatorial methods to approaches utilizing the theory of modular forms. We perform our analysis by extending the statistical mechanics inspired approaches in the case of partitions to the infinite family here.
\end{abstract}
\maketitle

\section{Introduction and statement of results}
In this article, we study the structure of a typical finite-dimensional representation of the Lie algebra $\mathfrak{sl}_{r+1}(\mathbb{C})$ with $r \geq 2$ if the dimension of the (not necessarily irreducible) representation is large.
More specifically, we investigate the typical features of an $n$-dimensional representation $\rho$ that is randomly sampled
from all $n$-dimensional representations
using the uniform measure
\begin{equation*}
P_n(\rho):=\frac{1}{p_r(n)}
\end{equation*}
with $p_{r}(n)$ denoting the number of $n$-dimensional representations of $\mathfrak{sl}_{r+1}(\mathbb{C})$. By Weyl's Theorem, any such representation $\rho$ decomposes as a direct sum of irreducible representations.
In turn, the finite-dimensional, irreducible representations are uniquely identified through dominant integral weights. More precisely, any $\k \in \IN^r$ parametrizes a dominant integral weight
\begin{equation*}
(k_1 - 1) \bm{\l}_1+ \ldots + (k_r - 1) \bm{\l}_r,
\end{equation*}
where $\bm{\l}_1, \ldots, \bm{\l}_r$ is a set of fundamental weights, and it identifies the (up to isomorphism unique) irreducible representation $\rho_{\k}$, which has this dominant integral weight as its highest weight (see e.g.~\cite{FH} for relevant background).
So we have
$$
\rho=\bigoplus_{\k \in \N^r} \rho_{\k}^{X_{\k}(\rho)},
$$
where each $X_{\k}(\rho)\in \N_0$ counts the multiplicity of $\rho_{\k}$ in $\rho$.
Note that the dimension of $\rho_{\k}$ is\footnote{Since $r$ is fixed in most of the paper, we often write $a(\k)=a_r(\k)$.}
\begin{equation}\label{E:akDef}
\dim (\rho_{\bm{k}})= \frac{1}{c_r} \prod_{1 \leq \ell \leq j\leq r}\left(k_\ell+\dots + k_j \right)=:a_r(\bm{k}),
\where
c_r:=r!(r-1)! \cdots 1! ,
\end{equation}
which is a homogenous polynomial in $\bm{k}$.
So if  $\dim(\rho)=n$, then the multiplicities of irreducible representations are related by
\begin{equation}\label{E:dimensionconstrain}
\dim(\rho)=\sum_{\k \in \N^r} a(\k)X_{\k}(\rho)=n
\end{equation}
and the generating function for $p_r (n)$ can be expressed as an infinite product
\begin{equation*}
\sum_{n \geq 0} p_{r}(n)q^n=\prod_{\bm{k}\in \mathbb{N}^r}\frac{1}{1-q^{a(\bm{k})}}.
\end{equation*}
This product is in fact the key to the recent work of Romik \cite{Romik}, who also pursued a probabilistic perspective and proved an asymptotic formula for $p_2 (n)$, the number of $n$-dimensional representations of $\mathfrak{sl}_{3}(n)$, undertaking a deep study of the {\it Witten zeta function} $\sum_{\bm n \in \N^2} a_2(\bm n)^{-s}$ in the process.\footnote{As a general reference for zeta functions of root systems, we refer the reader to \cite{KMT}.}  Work by Brindle, Franke, and two of the authors then extended Romik's asymptotic to a general family of enumeration functions generated by infinite products \cite{BBBF} (see also \cite{BF}).

In this paper, we view the $X_{\k}$ as random variables on the set of all finite-dimensional representations, and our goal is to understand the $n \to \infty$ asymptotic distributions
of various statistics built out of the $X_{\k}$ if the representation is sampled with $P_n$ (so that the $X_{\k}$ are subject to the constraint \eqref{E:dimensionconstrain}).
One can ask, for example, among the irreducible representations occurring in the decomposition of an $n$-dimensional representation, what is the largest dimension, or what is the maximum height (see Subsection \ref{S:Heightdistribution}), or how many irreducible representations occur, as $n \to \infty$.  All of these statistics may be simply described in terms of the multiplicities $X_{\k}$, and we develop a general method to derive distributions for such statistics.

If $r=1$, then we have $a_1(k)=k$, and a random $n$-dimensional representation of $\mathfrak{sl}_2(\mathbb{C})$ corresponds to an integer partition of $n$, in which the part $k$ occurs with multiplicity $X_k$.  Statistics for partitions of $n$ under the uniform measure were well-studied over the previous century by Erd\H{o}s, Lehner, Szalay, Tur\'{a}n, and many others \cite{EL,ErdosTuran,SzalayTuran1,SzalayTuran2,SzalayTuran3,Temperley}.  The work of Fristedt \cite{Fristedt} and its further developments by Pittel \cite{Pittel} then marked a significant advance in the toolkit for studying such statistics.
In Table \ref{Table:partnstats}, we provide a summary of such distributions derived in \cite{DemboVershikZeitouni,EL,Fristedt},
where we show the name of each statistic, its description in terms of the multiplicities $X_{k}$, the normalization taken to obtain a non-trivial distribution, and a description of the distribution along with the reference proving this distribution.
Here, we define
\begin{equation}\label{E:limitshapedef_r1}
A:=\frac{\sqrt{6}}{\pi}
\andd
f_1 (t) := - \log\left(1-e^{-t}\right) .
\end{equation}
Moreover, in the column ``distribution'', cdf means the cummulative distribution function 
and the notation $g(t) {\scriptstyle \overset{P}{\to}} f(t)$ means convergence in probability that is uniform in $t \geq \eta$ for any $\eta >0$.

\begin{table}[h!]
\centering
\vspace{-5pt}
\begin{tabular}{c | c | c | c | c}
name  & \makecell{description in \\ terms of $X_k$  } & normalization & distribution & Theorem  \\
\specialrule{.1em}{.0em}{.0em}
\makecell{multiplicity of \\ a small part}
&
\makecell{$X_{k_n}$ with \\ $k_n=o\left(\sqrt{n}\right)$}
&
$\frac{k_n}{A\sqrt{n}}X_{k_n}$
&
cdf: $1-e^{-x}$
&
\cite[Thm. 2.1]{Fristedt}
\\
\hline
\makecell{joint multiplicities \\ of small parts}
&
\makecell{$(X_{k})_{k \leq k_n}$ with \\ $k_n=o\!\left(n^\frac14\right)$}
&
$\left(\frac{k}{A\sqrt{n}}X_{k}\right)_{k \leq k_n}$
&
cdf: $\displaystyle\prod_{k \leq k_n} \lp 1-e^{-x_{k}} \rp$
&
\cite[Thm. 2.2]{Fristedt}
\\
\hline\rule{0pt}{4ex}
largest part
&
$Y_1:=\displaystyle \max_{X_{k}>0} (k)$
&
$\frac{Y_1-A\sqrt{n}\log\left(A\sqrt{n}\right)}{A\sqrt{n}}$
&
cdf: $e^{-e^{-x}}$
&
\cite[Thm. 1.1]{EL}
\\
\hline\rule{0pt}{4ex}
shape
&
$\varphi(t):=\sum_{k \geq t} X_{k}$
&
$A\sqrt{n}\, \varphi\!\left(\frac{1}{A\sqrt{n}}t\right)$
&
$A\sqrt{n} \,\varphi\!\left(\frac{t}{A\sqrt{n}}\right)\!\overset{P}{\to} f_1(t)$
&
\cite[Thm. 1]{DemboVershikZeitouni}
\\[.5em]
\hline\rule{0pt}{4ex}
\makecell{total number \\ of parts}
&
$N:=\sum_{k\geq 1} X_{k} $
&
$\frac{N-A\sqrt{n}\log\left(A\sqrt{n}\right)}{A\sqrt{n}}$
&
\makecell{cdf: $e^{-e^{-x}}$}
&
\cite[Thm. 1.1]{EL}
\\
\bottomrule
\end{tabular}
\caption{Distributions for partitions}
\label{Table:partnstats}
\end{table}

To describe Fristedt's methodology in more detail, the striking point here is a statistical mechanics-type approach to the problem, where one views the set of partitions as the state space of a physical system with the size of a partition corresponding to its energy. From this point of view, the uniform measure on partitions of a given size corresponds to a microcanonical ensemble on the said physical system. A phenomenon well-known to physicists as {\it equivalence of ensembles} then states that in the thermodynamic limit the microcanonical ensemble coincides with the so-called canonical ensemble, where temperature instead of energy is held fixed. For partitions, this means allowing the size of a partition to vary with a Boltzmann distribution that identifies partition sizes with energy. Crucially, the multiplicities $X_k$ become independent under the Boltzmann model, which makes it much simpler to work with than the unwieldy uniform measure. Fristedt then proved the equivalence of ensembles in this setup for a wide range of statistics, which allows identifying distributions deduced from the Boltzmann model with those obtained from the uniform measure under the limit $n \to \infty$.
This technique\footnote{This is also referred to as a conditioning device. A similar method is poissonization (see \cite{SL}).} has since become ubiquitous in analytic combinatorics more broadly \cite{ABT1,ABT2,DFLS}.

We adopt the same approach here by viewing finite-dimensional representations of $\mathfrak{sl}_{r+1}(\mathbb{C})$ as physical states (the dimension of a representation giving its energy) and introduce the corresponding Boltzmann model, which is a probability measure $Q_q$ on all finite-dimensional representations, defined as follows. For $q \in (0,1)$ and a finite-dimensional $\mathfrak{sl}_{r+1}(\mathbb{C})$-representation $\rho$, define
\begin{equation} \label{D:Boltzmannmodel}
    Q_q(\rho):=q^{\dim(\rho)}\prod_{\k \in \N^r}\left(1-q^{a(\k)}\right).
\end{equation}
\noindent To relate $Q_q$ to $P_n$, we first choose $q=q_n$ to maximize the probability, $Q_q(\dim = n)$, that we sample a representation of dimension $n$.  In Proposition \ref{P:SaddlePointAsympGrowth}, we show that
\begin{equation}\label{eq:equation_of_state}
q_n=e^{-s_n^{\frac{r(r+1)}{2}}}, \where  s_n \asymp n^{-\frac{2}{r(r+3)}}.
\end{equation}
Then Proposition \ref{P:EquivalenceofEnsembles}, which we refer to as \textit{equivalence of ensembles}, roughly states that if $q=q_n$, many limiting distributions under $Q_{q_n}$ and $P_n$ coincide. More precisely, we prove equivalence of ensembles for statistics that depend only on $X_{\k}$ with $\k$ very small or $\k$ very large in an appropriately defined sense (see Corollary \ref{C:EquivalenceofEnsembles}).
We then utilize this equivalence to solve our questions through a Boltzmann model and then transfer the results to our original questions holding the dimension fixed as in \eqref{E:dimensionconstrain}.
For some statistics, however, such as the shape or the total number of irreducible representations, we must deal with $X_{\k}$ involving $\k$ in the middle range not covered by this result.
In this case we employ what we call the \textit{rare events lemma} (Lemma \ref{lem:ExpSmallPrinciple}), which states that all events with ``small enough'' probability under the Boltzmann model have zero limiting probability under the uniform measure. So as long as the nontrivial contributions of $X_{\k}$ with $\k$ in the aforementioned middle range constitute such rare events, one can remove such contributions and again use the equivalence of ensembles stated above. Compared to the case of partitions,
a technical challenge in the case of $\mathfrak{sl}_{r+1}(\mathbb{C})$-representations arises from the need to estimate sparse multi-sums and products.
For example, we use Weyl differencing in an inductive fashion to handle such estimations and prove Proposition \ref{P:EquivalenceofEnsembles}.

\begin{rem}
The parameter $T := - {\scriptsize \tfrac{1}{\log (q)}}$ is the temperature of the given Boltzmann model. So the relation \eqref{eq:equation_of_state}
ensuring the equivalence of ensembles
corresponds to the physical equation of state ${\scriptsize E\asymp T^{\frac{r+3}{r+1}}}$
between the temperature $T$ and energy $E := n$.
\end{rem}

Our results are summarized in Table \ref{Table:slrstats}, where we list the 
theorem number containing the precise statement.
In the column ``distribution'', mgf means the moment generating function (i.e., $E(e^{uX})$ for a random variable $X$)
and the notation $g(\bm{t}) {\scriptstyle \overset{P}{\to}} f(\bm{t})$ denotes convergence in probability that is uniform on any set of the form $[\eta, \infty)^r$ for $\eta >0$.
Here, for $\bm{t}\in (0,\infty)^r$, we 
extend equation \eqref{E:limitshapedef_r1} and
define\footnote{Throughout we denote $\bm{dy}:=dy_1\cdots dy_r$.}
\begin{equation}\label{E:limitshapedef}
f_r(\bm{t}):=\int_{\prod_{j=1}^r[t_j,\infty)} \frac{e^{-a(\bm{y})}}{1-e^{-a(\bm{y})}}\bm{dy} .
\end{equation} 
The sequence $s_n$ is as in Proposition \ref{P:SaddlePointAsympGrowth} and the constants in the column ``normalization'' are
\begin{align*}
    a_n^{[D]} &:=s_n^{-\frac{r(r+1)}{2}}\left(\log\left(s_n^{-r}\right)-\frac{r-1}{r+1}\log\left(\log\left(s_n^{-r}\right)\right) +\log\left(\frac{2C_r}{r+1}\right)\right), \quad b_n^{[D]} := s_n^{-\frac{r(r+1)}{2}}, \\
    a_n^{[H]} &:=\frac{r!}{2}s_n^{-\frac{r+1}{2}}\a_n^{-\frac{r-1}{r}}\left(\frac{\a_n}{(r-1)!}
-\frac{r-1}{r} \log (\a_n) \right), \quad b_n^{[H]} :=\frac{r!}{2}s_n^{-\frac{r+1}{2}}\a_n^{-\frac{r-1}{r}},
\end{align*}
where 
\begin{equation}\label{eq:introduction_C_r_alpha_n_definitions}
C_r := \int_{\substack{\bm{y} \in \IR_{> 0}^r \\ a(\bm{y}) \leq 1}} \bm{dy}
\andd
\a_n := \GG (r+1)  \log \!\lp 2 \GG (r) s_n^{-\frac{r+1}{2}} \rp.
\end{equation}
Finally, the height of the representation is defined as the largest value of height among its weights (measured through the inner product with the Weyl vector). To display this in Table \ref{Table:slrstats} we define the vector
$\bm{1} := (1,1,\ldots,1)$ and
\begin{equation}\label{eq:height_l_function_defn}
L(\bm{k}):=\frac{1}{2} \sum_{1 \leq \ell \leq j \leq r} (k_\ell+k_{\ell+1}+\ldots+k_j)
=
\frac{1}{2} \sum_{j=1}^r j (r+1-j) k_j .
\end{equation}

\begin{table}[h]
\centering
\vspace{-5pt}
\begin{tabular}{c | c | c | c | c}
name  & \makecell{description in \\ terms of $X_{\k}$  } & normalization & distribution & Theorem  \\
\specialrule{.1em}{.0em}{.0em}
\makecell{multiplicity of \\ a small irreducible \\ representation }
&
\makecell{$X_{\k_n}$ with \\ $a(s_n\k_n)=o(1)$}
&
$a(s_n\k_n)X_{\k_n}$
&
cdf: $1-e^{-x}$
&
\ref{T:Mult} (1)
\\
\hline
\makecell{joint multiplicities \\ of small irreducible \\ representations }
&
\makecell{$(X_{\k})_{\k \in I_n}$ with \\ $a(s_n\k)=o\!\left(s_n^{\frac{r(r+1)}{r+3}}\right)$ \\ for $\k \in I_n$}
&
$\lp a(s_n\k)X_{\k} \rp_{\k \in I_n}$
&
cdf: $\displaystyle\prod_{\k \in I_n} \lp 1-e^{-x_{\k}} \rp$
&
\ref{T:Mult} (2)
\\
\hline
\makecell{largest dimension \\ among irreducible \\ representations }
&
$D:=\displaystyle\max_{X_{\k}>0} \!\lp a(\k) \rp$
&
$\frac{D-a_n^{[D]}}{b_n^{[D]}}$
&
cdf: $e^{-e^{-x}}$
&
\ref{T:lgstajk}
\\
\hline
height
&
$H:= \displaystyle\max_{X_{\k}>0} \!\lp L(\k-\bm{1}) \rp$
&
$\frac{H-a_n^{[H]}}{b_n^{[H]}}$
&
cdf: $e^{-e^{-x}}$
&
\ref{T:Height}
\\
\hline
shape
&
$\varphi(\bm{t}):=\sum_{\substack{\k \in \N^r \\ k_j \geq t_j}} X_{\k} $
&
$s_n^{r}\,\varphi(s_n^{-1}\bm{t})$
&
$s_n^{r}\,\varphi(s_n^{-1}\bm{t})\overset{P}{\to} f_r(\bm{t})$
&
\ref{T:Shape}
\\
\hline
\makecell{total number of \\ irreducible \\ representations}
&
$N:=\sum_{\k \in \N^r} X_{\k} $
&
$s_n^{\frac{r(r+1)}{2}} N$
&
\makecell{mgf: $\!\!\displaystyle\prod_{\k \in \N^r} \frac{1}{1-ua(\k)^{-1}} $ }  &
\ref{T:NumberofParts}
\\
\bottomrule
\end{tabular}
\caption{Main results}
\label{Table:slrstats}
\vspace{-5pt}
\end{table}

\begin{rem}
Integer partitions can be represented visually by their Young/Ferrer's diagrams, where each part corresponds to a row of cells.  A new partition is obtained by taking the columns as parts.  This involution swaps the largest part with the number of parts, which is why these rows in Table 1.1 have the same distribution.  In our case, $\mathfrak{sl}_{r+1}(\mathbb{C})$-representations for $r\geq 2$ lack such an involution, so we obtain the three distinct distributions in Table \ref{Table:slrstats} for the height, largest dimension, and total number of irreducible representations, which reduce to either one of these quantities for $r=1$.
\end{rem}

We also prove the following auxiliary asymptotic for the partial sums of the number of irreducible representations of a given dimension defined as
\begin{equation*}
R_r(x):=\sum_{m \leq x} \varrho_r(m)
\where
\varrho_r(m):=\#\{\k \in \IN^r : a(\k)=m\}.
\end{equation*}
\begin{prop}\label{prop:irrep_partial_sum_asymptotics}
For $r\in\N$ and $x>0$ we have
\begin{equation*}
R_r(x) = C_r x^{\frac{2}{r+1}} + O_r \!\lp x^{\frac{2 (r-1)}{r^2}} \rp.
\end{equation*}
\end{prop}

The paper is organized as follows. In Section \ref{S:PartialSumIrreps} we prove Proposition \ref{prop:irrep_partial_sum_asymptotics}. In Section \ref{S:MinorArcs} and the accompanying Appendix \ref{A:LowerBounds}, we give some elementary results on how the dimensions of irreducible representations $a (\bm{k})$ distribute themselves over circles of varying size using Weyl differencing.
These are then used as minor arc bounds in Section \ref{S:Boltzmann}, where we introduce
the Boltzmann model, its basic properties, and prove equivalence of ensembles (Proposition \ref{P:EquivalenceofEnsembles}) employing the Circle Method.
Equivalence of ensembles is the essential tool that we employ in subsequent sections to allow transfer of distributions from the Boltzmann model to the uniform measure.
We use this and prove our first main result on the distribution of the multiplicities for small irreducible representations in Section \ref{S:SmallMultiplicities}.
In Section \ref{S:MaxDimHeight}, we show the distribution for the maximum dimension and height. Then in Section \ref{S:LimitShapeTotalIrreps}, we prove distributions for the total number of irreducible representations and the shape.
Finally in Section \ref{S:Conclusion}, we collect some open problems and avenues for further exploration.

\section*{Acknowledgements}
The second author has received funding from the European Research Council (ERC) under the European Union’s Horizon 2020 research and innovation programme (grant agreement No. 101001179).
The third author was supported by the SFB/TRR 191 “Symplectic Structure in Geometry, Algebra and Dynamics”, funded by the DFG (Projektnummer 281071066 TRR 191).
We thank Dan Romik for sharing this problem with us, and we thank Siu Hang Man for useful discussions on Weyl differencing.

\section{The Average Number of Irreducible Representations of Dimension Less Than $x$}\label{S:PartialSumIrreps}
We begin by proving Proposition \ref{prop:irrep_partial_sum_asymptotics}, a preliminary result that gives an asymptotic expansion for the average number of irreducible representations of dimension less than $x$.
For this purpose, we introduce in this section the polynomial
\begin{equation*}
h_r(\bm{k}) :=
\prod_{1 \leq \ell \leq j\leq r}\left(k_\ell+\dots + k_j \right),
\end{equation*}
which is homogeneous of degree $\frac{(r+1)r}{2}$ and
which up to the constant factor $c_r$ gives the dimension of the irreducible representation $\rho_{\bm{k}}$ as in \eqref{E:akDef}. We start with a technical lemma.

\begin{lem}\label{lem:irrep_polynomial_bound}
For integers $r \geq 2$ and $1 \leq a \leq r$ and $\bm{y} = (y_1,\ldots,y_r) \in \IR_{\geq 0}^r$, we have
\begin{equation*}
h_r(\bm{y})^{r-1} \geq y_a^\frac{r(r-1)}{2} h_{r-1}(\bm{y}_{\widehat{a}})^r
\end{equation*}
where $\bm{y}_{\widehat{a}} \in \IR_{\geq 0}^{r-1}$ is the vector $\bm{y}$ with $y_a$ omitted.\footnote{Note that both sides of the proposed inequality are homogeneous polynomials of degree $\frac{1}{2} (r+1)r(r-1)$.}
\end{lem}
\begin{proof}
We start by writing
\begin{equation*}
h_r(\bm{y}) = \prod_{1 \leq \ell \leq j\leq r} y_{\ell,j}
\andd
h_{r-1}(\bm{y}_{\widehat{a}}) =
\prod_{1 \leq \ell \leq j < a} y_{\ell,j}
\prod_{a<\ell\leq j \leq r} y_{\ell,j}
\prod_{1 \leq \ell < a < j \leq r} \widehat{y}_{\ell,j},
\end{equation*}
where
\begin{equation*}
y_{\ell,j} :=  y_\ell + y_{\ell+1} + \ldots + y_j
\andd
\widehat{y}_{\ell,j} :=  y_\ell + \ldots + y_{a-1} + y_{a+1} + \ldots + y_j .
\end{equation*}
By dropping the term $y_a$ we have $y_{\ell,j} \geq \widehat{y}_{\ell,j}$ for any $1 \leq \ell < a < j \leq r$, which then leads to
\begin{equation*}
h_r(\bm{y}) \geq
y_a
h_{r-1}(\bm{y}_{\widehat{a}})
\prod_{1 \leq \ell < a} y_{\ell,a}
\prod_{a<j\leq r} y_{a,j}.
\end{equation*}
So the lemma statement follows if we can prove that
\begin{equation}\label{eq:irrep_polynomial_bound_lem_reduced}
\prod_{1 \leq \ell < a} y_{\ell,a}^{r-1}
\prod_{a<j\leq r} y_{a,j}^{r-1}
\geq
y_a^{\frac{1}{2} (r-1)(r-2)}  h_{r-1}(\bm{y}_{\widehat{a}}) .
\end{equation}
As above we can drop terms and bound
\begin{equation*}
y_{\ell,a}^{r-1} \geq
y_{\ell,a}^{r-a}  y_a^{\ell-1} \prod_{j=\ell}^{a-1} y_{\ell,j}
\mbox{ for any } 1 \leq \ell < a
\ \mbox{ and } \
y_{a,j}^{r-1} \geq
y_{a,j}^{a-1}  y_{a}^{r-j} \prod_{\ell=a+1}^{j} y_{\ell,j}
\mbox{ for any } a<j\leq r .
\end{equation*}
Combining these two facts we find
\begin{equation*}
\prod_{1 \leq \ell < a} y_{\ell,a}^{r-1}
\prod_{a<j\leq r} y_{a,j}^{r-1}
\geq
y_a^{\frac{1}{2} (a-1)(a-2) + \frac{1}{2} (r-a)(r-a-1)}
\prod_{1 \leq \ell \leq j < a} y_{\ell,j}
\prod_{a<\ell\leq j \leq r} y_{\ell,j}
\prod_{1 \leq \ell < a < j \leq r} y_{\ell,a} y_{a,j}
\end{equation*}
including the edge cases $a=1$ and $a=r$.
Noting that the terms in the final product can be bounded as
$y_{\ell,a} y_{a,j} \geq y_a \widehat{y}_{\ell,j}$, the inequality \eqref{eq:irrep_polynomial_bound_lem_reduced} follows.
\end{proof}

Our next task is to use Lemma \ref{lem:irrep_polynomial_bound} to bound (and show the convergence of) certain integrals that appear in estimating the average number of irreducible representations.
\begin{lem}\label{lem:irrep_polynomial_integral}
For $\ve > 0$, $r\in\N$, and $1 \leq a \leq r$, we have
\begin{equation*}
I_r := \int_{\substack{\bm{y} \in \IR_{> 0}^r \\ h_r(\bm{y}) \leq 1}} \bm{dy} < \infty
\andd
I_{r,a,\ve} :=
\int_{\substack{\bm{y} \in \IR_{> 0}^r \\ h_r(\bm{y}) \leq 1 \\ y_a \leq \ve}} \bm{dy}
\ll_r \ve^{\frac{1}{r}}.
\end{equation*}
\end{lem}
\begin{proof}
The lemma statement trivially holds for $r=1$ and we apply induction for $r \geq 2$.
To check the convergence of the improper integral $I_r$, we need to check that the boundary contributions as $y_a \to 0^+$ are finite.\footnote{Note that $h_r (\bm{y})$ tends to zero near those boundaries, so the integration region is tending to infinity.}
This is in turn assured by our claim for $I_{r,a,\ve}$, which we prove next.
For this purpose, we use  Lemma \ref{lem:irrep_polynomial_bound} to enlarge the integration region and get the upper bound
\begin{equation*}
I_{r,a,\ve} \leq
\int_{\substack{\!\!\!\!\bm{y} \in \IR_{> 0}^r, \ y_a \leq \ve
\\ h_{r-1}(\bm{y}_{\widehat{a}})  \leq  y_a^{-\frac{r-1}{2}}}} \bm{dy} .
\end{equation*}
Since $h_{r-1}(\bm{y}_{\widehat{a}})$ is homogeneous of degree $\frac{r(r-1)}{2}$, we can rescale the vector $\bm{y}_{\widehat{a}}$ as $\bm{y}_{\widehat{a}} \mapsto y_a^{-\frac{1}{r}}  \bm{y}_{\widehat{a}}$ to find
$I_{r,a,\ve} \leq r  I_{r-1} \ve^{\frac{1}{r}}$.
\end{proof}

\begin{rem}
Since $a(\k) = \tfrac{1}{c_r} h_r(\k)$, the constant $C_r$ defined in \eqref{eq:introduction_C_r_alpha_n_definitions} satisfies
\begin{equation}\label{eq:Cr_Ir_relation}
C_r = I_r c_r^{\frac{2}{r+1}} < \infty .
\end{equation}
\end{rem}

Next we use Lemma \ref{lem:irrep_polynomial_integral} to find bounds on the number of lattice points bounded by $h_r$.
\begin{lem}\label{lem:irrep_polynomial_lattice_counting}
For $r\in\N$ and $x \in \IR_{> 0}$ we have
\begin{equation*}
\sum_{\substack{\bm{k} \in \IN^r \\ h_r (\bm{k}) \leq x}} 1
= I_r  x^{\frac{2}{r+1}} + O_r \!\lp x^{\frac{2 (r-1)}{r^2}} \rp.
\end{equation*}
\end{lem}
\begin{proof}
Note that $h_r (\bm{y})$ is nondecreasing in each of its variables since we have a polynomial with nonnegative coefficients, so we have
\begin{equation*}
\int_{\substack{\bm{y} \in \IR_{\geq 1}^r \\ h_r(\bm{y}) \leq x}} \bm{dy}
\leq
\sum_{\substack{\bm{k} \in \IN^r \\ h_r (\bm{k}) \leq x}} 1
\leq
\int_{\substack{\bm{y} \in \IR_{> 0}^r \\ h_r(\bm{y}) \leq x}} \bm{dy} .
\end{equation*}
Combining these two bounds and recalling that $h_r$ is a homogeneous of degree $\frac{r(r+1)}{2}$, we find
\begin{equation*}
0 \leq
x^{\frac{2}{r+1}}
\int_{\substack{\bm{y} \in \IR_{> 0}^r \\ h_r(\bm{y}) \leq 1}} \bm{dy}
-
\sum_{\substack{\bm{k} \in \IN^r \\ h_r (\bm{k}) \leq x}} 1
\leq
x^{\frac{2}{r+1}}
\sum_{a=1}^r
\int_{\substack{\bm{y} \in \IR_{> 0}^r, \  h_r(\bm{y}) \leq 1 \\ y_a \leq \ve (x) }}
\bm{dy}
\quad \mbox{with } \ve (x) := x^{-\frac{2}{r(r+1)}}.
\end{equation*}
The result then follows by Lemma \ref{lem:irrep_polynomial_integral}.
\end{proof}

We are now ready to prove Proposition \ref{prop:irrep_partial_sum_asymptotics}.
\begin{proof}[Proof of Proposition \ref{prop:irrep_partial_sum_asymptotics}]
With $R_r(x) \! = \!\!\!\!\!\!\!\displaystyle\sum_{\substack{\bm{k} \in \IN^r \\ h_r (\bm{k}) \leq c_r x}} \!\!\!\!\!\! 1$, the result follows from  \eqref{eq:Cr_Ir_relation} and Lemma \ref{lem:irrep_polynomial_lattice_counting}.
\end{proof}

\section{Irreducible Representation Dimensions on a Circle}\label{S:MinorArcs}
In Section \ref{S:Boltzmann}, we establish equivalence of ensembles using the Circle Method. Our goal in this section is to establish the lower bounds we use in our minor arc estimates through an inductive application of Weyl differencing type arguments. We state these results separately as preliminary facts on the distribution of the numbers $a_r (\bm{k})$ on a circle.\footnote{
Since we apply induction over the rank $r$, we restore the subscript in $a_r (\bm{k})$ within this section.
}

Our arguments are based on the point that the polynomial $a_r (\bm{k})$
satisfies
\begin{equation*}
a_r (\bm{k}) = \frac{a_{r-1} (\bm{k}_{\widehat{r}})}{r!}
\prod_{1 \leq j \leq r} (k_j + \ldots + k_r),
\end{equation*}
where, as in Section \ref{S:PartialSumIrreps}, we write $\bm{y}_{\widehat{a}} \in \IR_{\geq 0}^{r-1}$ for the vector $\bm{y}$ with $y_a$ omitted.  Note that the second factor above includes all the terms that contain $k_r$.
In particular, in terms of the indeterminate $k_r$,
this is a polynomial of degree $r$ with the leading coefficient $\frac{a_{r-1} (\bm{k}_{\widehat{r}})}{r!}$ and this is what leads to our inductive step. To extract this leading coefficient, we define the {\it finite difference operator} 
$$\DD_h(f(\bm{k})) := f(\bm{k}_{\widehat{r}}, k_r+h) - f(\bm{k}_{\widehat{r}}, k_r)$$
so that
\begin{equation*}
\DD_h^r(f(\bm{k}))
= \sum_{m=0}^r (-1)^{m+r} \mat{r \\ m} f(\bm{k}_{\widehat{r}}, k_r+mh)
\andd
\DD_h^r(a_r(\bm{k})) = a_{r-1} (\bm{k}_{\widehat{r}}) h^r .
\end{equation*}

\begin{prop}\label{prop:sinsquare_sum_lower_bound_general_r_case}
Let $r \in \mathbb{N}_{\geq 2}$, $0 < \ve \leq \frac{1}{32}$, and $\nu_r := \frac{r(r+1)}{2}$.
Then for any integer $N \geq 4$
and any real $\th$ with
$\ve N^{-\nu_r} \! \leq \th \leq \frac{1}{2}$ we have
\begin{equation*}
\# \left\{ \bm{k} \in \LL_{N,r} \!:
2^{-\nu_r} \ve < \{\th a_r (\bm{k}) \} < 1 - 2^{-\nu_r} \ve \right\} \geq \frac{N^r}{32} ,
\end{equation*}
where
\begin{equation*}
\LL_{N,r} := \{ \bm{k} \in \IZ^r \!: N \leq k_j \leq (j+2)N \mbox{ for } j \in \{1,2,\ldots,r\} \}.
\end{equation*}
\end{prop}
\begin{proof}
We apply induction over $r$. The base case $r=2$ is the content of Lemma \ref{lem:a2_lower_bound_window_ladder3}. So let $r \geq 3$ and assume that the proposition statement holds if $r$ is replaced with $r-1$.
Our goal is to prove the proposition statement for two separate ranges, namely $\ve N^{-\nu_{r-1}}\leq\th\leq\frac{1}{2}$ and $\ve N^{-\nu_r}\leq\th\leq\frac{1}{2N^r}$ with the latter equivalent to $\ve N^{-\nu_{r-1}} \! \leq N^r \th \leq \tfrac{1}{2}$ because $\nu_r = \nu_{r-1} + r$.
Note that with $r \geq 3$, these two ranges cover the entire interval assumed in the proposition.\footnote{
Note that this statement does not hold for $r=2$, which is why we start with the base case $r=2$ instead of $r=1$.
}

So we first assume that $\ve N^{-\nu_{r-1}} \! \leq \th \leq \frac{1}{2}$ and focus on the subrange $N \leq k_r < (r+2) N$ to rewrite such $k_r$ uniquely as
\begin{equation*}
k_r = N + (r+1)\ell + m
\quad \mbox{with }
\ell \in \{0,1,\ldots,N-1\}
\mbox{ and }
m \in \{0,1,\ldots,r\}.
\end{equation*}
This allows us to lower bound $\# \left\{ \bm{k} \in \LL_{N,r} \!:
2^{-\nu_r} \ve < \{\th a_r (\bm{k}) \} < 1 - 2^{-\nu_r} \ve \right\}$ by
\begin{multline}
\# \big\{ \bm{k}_{\widehat{r}} \in \LL_{N,r-1}, \ell \in \{0,\ldots,N-1\} \!:
2^{-\nu_r} \ve < \{\th a_r (\bm{k}_{\widehat{r}}, N + (r+1)\ell + m) \}
< 1 - 2^{-\nu_r} \ve
\\
\mbox{ for at least one } m \in \{0,\ldots,r\}
\big\}.\!\!\!\!
\label{eq:prop:sinsquare_sum_lower_bound_step1}
\end{multline}
Next note that for any $\bm{k}_{\widehat{r}} \in \LL_{N,r-1}$ and $\ell \in \{0,\ldots,N-1\}$, if we have
\begin{equation*}
2^{-\nu_{r-1}} \ve < \{\th \Delta_1^r(a_r (\bm{k}_{\widehat{r}}, N + (r+1)\ell)) \}
< 1 - 2^{-\nu_{r-1}} \ve,
\end{equation*}
then we find
\begin{equation*}
2^{-\nu_r} \ve < \{\th a_r (\bm{k}_{\widehat{r}}, N + (r+1)\ell + m) \}
< 1 - 2^{-\nu_r} \ve
\mbox{ for at least one } m \in \{0,\ldots,r\}.
\end{equation*}
This is because the numbers $\th a_r (\bm{k}_{\widehat{r}}, N + (r+1)\ell + m)$ can not be all close to integers if the finite difference they sum up to is not close to an integer.
This fact allows us to place a further lower bound on \eqref{eq:prop:sinsquare_sum_lower_bound_step1} as
\begin{equation*}
\# \left\{ \bm{k}_{\widehat{r}} \in \LL_{N,r-1}, \ell \in \{0,\ldots,N-1\} \!:
2^{-\nu_{r-1}} \ve < \{\th\Delta_1^r \left(a_r (\bm{k}_{\widehat{r}}, N + (r+1)\ell)\right) \}
< 1 - 2^{-\nu_{r-1}} \ve  \right\}.
\end{equation*}
Since $\Delta_1^r(a_r(\bm{k}_{\widehat{r}}, N + (r+1)\ell)) =
a_{r-1} (\bm{k}_{\widehat{r}})$ with no $\ell$-dependence, this equals
\begin{equation*}
N \, \# \!\left\{ \bm{k}_{\widehat{r}} \in \LL_{N,r-1} \!:
2^{-\nu_{r-1}} \ve < \{\th a_{r-1} (\bm{k}_{\widehat{r}}) \}
< 1 - 2^{-\nu_{r-1}} \ve  \right\}
\end{equation*}
and we use the inductive hypothesis to confirm the proposition statement. The argument in the other range $\ve N^{-\nu_{r-1}} \! \leq N^r \th \leq \frac{1}{2}$ is
similar: We use $\DD_N^r$, which extracts the factor of $a_{r-1} (\bm{k}_{\widehat{r}})$ needed in the inductive step while also producing the $N^r$ factor appearing in the range of $\th$.
\end{proof}

We state the version that we employ in Section \ref{S:Boltzmann} as a corollary.
\begin{cor}\label{cor:sinsquare_sum_lower_bound}
Let $r \!\in\! \mathbb{N}_{\geq 2}$ and $0 \!< \! \ve \! \leq \! \frac{1}{32}$.
For any $N \!\in\! \IN_{\geq 4}$ and $\th \!\in\! \IR$ with
$\ve N^{-\nu_r} \!\! \leq \! \th \! \leq \! \frac{1}{2}$ we have
\begin{equation*}
\sum_{\bm{k} \in \LL_{N,r}} \sin^2 (\pi  \th a_r(\bm{k})) \geq
C_{r,\varepsilon} N^r
\quad \mbox{with } C_{r,\varepsilon} := \frac{\sin^2 \!\lp 2^{-\nu_r} \pi \ve \rp}{32} .
\end{equation*}
\end{cor}

\section{The Boltzmann model}\label{S:Boltzmann}

\subsection{Basic properties}
Recall the definition of the Boltzmann model in \eqref{D:Boltzmannmodel}.  The following  proposition lists some useful properties; its proof is elementary and is omitted.
\begin{prop}\label{P:BoltzmannSimpleProperties} For any $q \in (0,1)$, the following properties of the Boltzmann model hold.
\begin{enumerate}[label=\rm(\arabic*),leftmargin=*]
\item We have that $Q_q$ is a probability measure on the set of all finite-dimensional $\mathfrak{sl}_{r+1}(\mathbb{C})$-representations including the ``empty representation" of dimension zero.
\item The $X_{\k}$ are independent under $Q_q$.  Moreover, they are geometrically distributed:
$$
Q_q(X_{\k}=\ell)=q^{\ell a({\k})}\left(1-q^{a({\k})}\right), \where \ell \in \mathbb{N}_0.
$$
\item If conditioned on the event $\dim=n$, then $Q_q$ agrees with the uniform measure.  That is,
$$
Q_q( \cdot | \dim=n)=P_n(\cdot).
$$
\end{enumerate}
\end{prop}

\begin{rem}
 By Proposition \ref{P:BoltzmannSimpleProperties} the probability space of finite-dimensional representations is equivalent to the space generated by sequences $(X_{\k})_{\k\in \mathbb{N}^r}$ of independent random variables with densities as above and such that all but finitely many of the $X_{\k}$ are zero.  In fact, the two are still equivalent without the assumption that finitely many $X_{\k}$ are non-zero.  This follows because the probability of the latter event is 0; that is, for any $q\in (0,1)$ and any infinite set $\{\k_j\}_{j \geq 1} \subset \N^r$, we have
 $$
Q_q(X_{\k_j}>0 \ \text{for} \ j \geq 1)=0.
$$
The proof of this fact is a straightforward application of the Borel--Cantelli Lemma.
\end{rem}

Let $\mathrm{E}_q$ and $\mathrm{Var}_q$ denote the expectation and variance operators under $Q_q$.  We choose $q=q_n$ to satisfy the saddle-point equation, which we write below in two equivalent forms.
\begin{prop}\label{P:SaddlePointexist}
For $n \in \mathbb{N}$, there exists a unique $q_n \in (0,1)$ satisfying the equivalent equations:
\begin{equation*}
Q_{q_n}(\dim=n)=\sup_{q \in (0,1)} Q_{q}(\dim=n)
\quad \Leftrightarrow \quad
\mathrm{E}_{q_n}(\dim)=n.
\end{equation*}
\end{prop}
\begin{proof}
By definition of the Boltzmann model, we have
$$
Q_{q_n}(\dim=n)=\! \sup_{q \in (0,1)} \! Q_{q}(\dim=n) 
\ \  \Leftrightarrow \ \ 
q_n^{-n}\prod_{\k \in \N^r} \left(1-q_n^{a(\k)}\right)^{-1}=\inf_{q \in (0,1)}\! q^{-n}\prod_{\k \in \N^r} \left(1-q^{a(\k)}\right)^{-1}.
$$
The $q$-series $\prod_{\k \in \N^r} (1-q^{a(\k)})^{-1}$ has non-negative coefficients, so $q \mapsto q^{-n}\prod_{\k \in \N^r} (1-q^{a(\k)})^{-1}$ is convex for $q \in (0,1)$ by \cite[p. 550, VIII.4] {FS}, and tends to $\infty$ as $q \to 0^+$ and as $q \to 1^-$, so has a unique minimum, at $q_n$ say.  A straightfoward calculation, taking the logarithmic derivative of the product and noting that
$$
\mathrm{E}_{q_n}(\dim) = \sum_{\k \in \N^r} \frac{a(\k)q_n^{a(\k)}}{1-q_n^{a(\k)}},
$$
gives the claimed equivalent forms of the above equation in the proposition.
\end{proof}

Next we give asymptotic estimates for the saddle point $q_n$ and the variance
\begin{equation*}
\sigma_n^2:=\mathrm{Var}_{q_n}(\dim)
= \sum_{\k\in \mathbb{N}^r}\!\frac{a(\k)^2q_n^{a(\k)}}{\left(1-q_n^{a(\k)}\right)^{\!2}} .
\end{equation*}

\begin{prop}\label{P:SaddlePointAsympGrowth}
    Let $q_n$ be as in Proposition \ref{P:SaddlePointexist}, define $s_n \in \IR_{>0}$ by $q_n=:\exp(-s_n^{\frac{r(r+1)}{2}})$, and let
    $$
    \mathcal{C}_r:=\left(\int_{[0,\infty)^r} \frac{a(\bm{t})e^{-a(\bm{t})}}{1-e^{-a(\bm{t})}} \bm{dt} \right)^{\frac{2}{r(r+3)}}, \quad \mathcal{C}_{r,\mathrm{var}}:=\int_{[0,\infty)^r} \frac{a(\bm{t})^2e^{-a(\bm{t})}}{(1-e^{-a(\bm{t})})^2}  \bm{dt}, \quad \mathcal{D}_{r,\mathrm{var}} := \frac{\mathcal{C}_{r,\mathrm{var}}}{\mathcal{C}_r^{r(r+2)}}.
$$
Then the involved integrals converge and we have the following asymptotic estimates:
\begin{enumerate}[label=\rm(\arabic*),leftmargin=*]
\item For the saddle-point, or equivalently the energy--temperature relation for the Boltzmann model,
\begin{equation*}
s_n = \mathcal{C}_r n^{-\frac{2}{r(r+3)}}
    \lp 1 + O \!\lp n^{-\frac{2}{r^2(r+3)}} \rp \rp .
\end{equation*}

\item For the variance $\sigma_n^2$, we have
\begin{equation*}
\sigma_n^2 = \mathcal{C}_{r,\mathrm{var}}s_n^{-r(r+2)}
    \lp 1 + O \!\lp s_n^{\frac{1}{r}} \rp \rp = \mathcal{D}_{r,\mathrm{var}} n^{\frac{2(r+2)}{r+3}}
    \lp 1 + O \!\lp n^{-\frac{2}{r^2(r+3)}} \rp \rp .
\end{equation*}

\item We have the third moment estimate
\begin{equation*}
\sum_{\k\in \mathbb{N}^r}\frac{a(\k)^3q_n^{a(\k)}}{\left(1-q_n^{a(\k)}\right)^3}
=O\!\left(s_n^{-\frac{r(3r+5)}{2}} \right) .
\end{equation*}
\end{enumerate}
\end{prop}

\begin{rem}
    Note that for $r=2$ we have
    $$- \log (q_n) = s_n^{3}
    = \mathcal{C}_2 n^{-\frac{3}{5}} \lp 1+ O \!\lp n^{-\frac{1}{10}} \rp \rp.
    $$
	This is consistent with \cite[Lemma 8.1]{Romik}, which states that $- \log (q_n) = a_1  n^{-\frac{3}{5}} - a_2  n^{-\frac{7}{10}} + ...\ $.
\end{rem}

\begin{proof}[Proof of Proposition \ref{P:SaddlePointAsympGrowth}]
    Using the characterization of $q_n \in (0,1)$ in Proposition \ref{P:SaddlePointexist} and the fact that $ a(\bm{t})$ is a degree $\frac{r(r+1)}{2}$ homogeneous polynomial, we have
    $$
    n=\sum_{\k \in \N^r} \frac{a(\k)q_n^{a(\k)}}{1-q_n^{a(\k)}}=s_n^{-\frac{r(r+3)}{2}}\sum_{\k \in \N^r} \frac{a(s_n\k)e^{-a(s_n\k)}}{1-e^{-a(s_n\k)}}s_n^r.
    $$
    The above implies that $\lim_{n \to \infty} q_n = 1$ and thus  $\lim_{n \to \infty} s_n = 0$.
The function $t \mapsto \frac{te^{-t}}{1-e^{-t}}$ is positive for $t>0$, is $O(1)$ as $t \to 0^+$, and has exponential decay as $t \to \infty$. So by Lemma \ref{lem:irrep_polynomial_integral} we have
\begin{equation}\label{eq:exp_rdim_integral_conv}
\int_{[0,\infty)^r} \frac{a(\bm{t})e^{-a(\bm{t})}}{1-e^{-a(\bm{t})}} \bm{dt}
\leq
\sum_{k\ge0}
\sup_{t \in [k,k+1]} \!\!\lp \frac{te^{-t}}{1-e^{-t}} \rp
\int_{\substack{\bm{y} \in \IR_{> 0}^r \\ k \leq a(\bm{y}) \leq k+1}} \bm{dy}
< \infty
\end{equation}
and the leftmost integral converges.
Since $a(\bm{t})$ has positive coefficients and $t \mapsto \frac{te^{-t}}{1-e^{-t}}$ is decreasing, the integral comparison criterion gives
$$
\int_{[s_n,\infty)^r} \frac{a(\bm{t})e^{-a(\bm{t})}}{1-e^{-a(\bm{t})}} \bm{dt} \leq \sum_{\k \in \N^r} \frac{a(s_n\k)e^{-a(s_n\k)}}{1-e^{-a(s_n\k)}}s_n^r \leq \int_{[0,\infty)^r} \frac{a(\bm{t})e^{-a(\bm{t})}}{1-e^{-a(\bm{t})}} \bm{dt}
$$
and hence
\begin{equation*}
0 \leq \int_{[0,\infty)^r} \frac{a(\bm{t})e^{-a(\bm{t})}}{1-e^{-a(\bm{t})}} \bm{dt} -
\sum_{\k \in \N^r} \frac{a(s_n\k)e^{-a(s_n\k)}}{1-e^{-a(s_n\k)}}s_n^r
\leq
\sum_{a=1}^r
\int_{\substack{\bm{t} \in \IR_{\geq 0}^r \\ t_a \leq s_n}}
\frac{a(\bm{t})e^{-a(\bm{t})}}{1-e^{-a(\bm{t})}} \bm{dt}
\ll s_n^{\frac{1}{r}}.
\end{equation*}
Here the last bound again follows from Lemma \ref{lem:irrep_polynomial_integral} and dividing the integration region to the $k \leq a(\bm{y}) \leq k+1$ regions as in \eqref{eq:exp_rdim_integral_conv}. Thus,
$$
n = s_n^{-\frac{r(r+3)}{2}}
\lp
\int_{[0,\infty)^r} \frac{a(\bm{t})e^{-a(\bm{t})}}{1-e^{-a(\bm{t})}} \bm{dt}
+ O \!\lp s_n^{\frac{1}{r}} \rp \rp
$$
and part (1) of the proposition follows.
The proofs of parts (2) and (3) are similar.
\end{proof}

\subsection{Equivalence of ensembles}
For probability measures $\mu_1,\mu_2$ on $\IN_0^d$ with $d \in \IN \cup \{\infty\}$
and events from the power set,\footnote{
Here $\IN_0^\infty$ denotes the set of sequences in $\IN_0$, where all but finitely many entries are zero.
}
define the {\it total variation metric}
\begin{equation*}
d_{\mathrm{TV}}(\mu_1,\mu_2):=\sup_{B \subset \IN_0^d} |\mu_1(B)-\mu_2(B)|.
\end{equation*}
In order for distributions under the Boltzmann model to coincide with those under $P_n$, we use the following transfer principle.

\begin{prop}\label{P:EquivalenceofEnsembles}
Let $\bm{X}_{I_n}:=\left(X_{\k}\right)_{\k \in I_n}$, where
$I_n \subset \mathbb{N}^r$ is such that
\begin{equation}\label{E:BoldXVarSmall}
\mathrm{Var}_{q_n}\!\left(\sum_{\k\in I_n} a(\k)X_{\k}\right) =
o\!\left(\sigma_n^2\right).
\end{equation}
Then
the probability measures
$Q_{q_n}\!\!\left(\bm{X}_{I_n}^{-1}\right)$
and $P_n\!\left(\bm{X}_{I_n}^{-1}\right)$ on $\IN_0^{d_n}$ with $d_n := |I_n|$ satisfy
\begin{equation}\label{eq:total_variation_pn_qn}
\lim_{n \to \infty} d_{\mathrm{TV}}\!\left(Q_{q_n}\!\!\left(\bm{X}_{I_n}^{-1}\right),
P_n\!\left(\bm{X}_{I_n}^{-1}\right)\right)=0.
\end{equation}
Moreover, we have
\begin{equation}\label{eq:Qqn_prob_dimension}
Q_{q_n}(\dim=n)\sim \frac{1}{\sqrt{2\pi \sigma_n^2} }.
\end{equation}
\end{prop}
\begin{proof}
Following \cite[Lemma 4.2]{Fristedt}, to prove \eqref{eq:total_variation_pn_qn} it suffices to find a sequence of sets $B_n \!\!\subset\! \IN_0^{d_n}$ with
\begin{equation}\label{E:XinBnAS}
\lim_{n \to \infty} Q_{q_n}(\bm{X}_{I_n}\in B_n) =1,
\end{equation}
and, uniformly for $\bm{x}_n\in B_n$,
\begin{equation}\label{E:QuotientUniformto1}
\lim_{n \to \infty} \frac{Q_{q_n}(\dim=n | \bm{X}_{I_n}=\bm{x}_n)}{Q_{q_n}(\dim=n)} =1,
\end{equation}
because for any $B \subset \IN_0^{d_n}$ one has the bound 
\begin{multline*}
| Q_{q_n}(\bm{X}_{I_n} \in B)-P_n(\bm{X}_{I_n} \in B) |
\\
\leq
Q_{q_n}(\bm{X}_{I_n} \in B_n^c)+\sum_{\bm{x} \in B_n} Q_{q_n}\!\left(\bm{X}_{I_n}=\bm{x}\right) \left| 1- \frac{Q_{q_n}(\dim = n | \bm{X}_{I_n} =\bm{x})}{Q_{q_n}(\dim=n)}\right|.
\end{multline*}
Our goal is to show that these two properties hold with the sets
\begin{equation*}
B_n:=\left\{(x_{\k})_{\k \in I_n} \in \IN_0^{d_n}: \left|E_{q_n}\!\left(\sum_{\k \in I_n} a(\k)X_{\k}\right)-\sum_{\k \in I_n} a(\k)x_{\k} \right| \leq c_n\right\},
\end{equation*}
where $c_n$ is an additional sequence of positive numbers chosen, using equation \eqref{E:BoldXVarSmall}, to satisfy
\begin{equation}\label{E:BoldXVarSmall2}
\s_n n^{-\frac{1}{3(r+3)}}  \leq
c_n=o\!\left(\sigma_n\right)
\ \mbox{ and } \
\mathrm{Var}_{q_n}\!\left(\sum_{\k \in I_n} a(\k)X_{\k}\right)=o\!\left(c_n^2\right).
\end{equation} Here, \eqref{E:XinBnAS} follows immediately from Chebyshev's inequality and \eqref{E:BoldXVarSmall2} as
$$
Q_{q_n}(\bm{X}_{I_n}\in B_n^c) \leq c_n^{-2} \, \mathrm{Var}_{q_n}\!\!\left(\sum_{\k \in I_n} a(\k)X_{\k}\right) = o(1).
$$
We  prove \eqref{E:QuotientUniformto1} by using the Circle Method, as in \cite{BB}.\footnote{Alternatively, it could be proved via Fourier inversion of characteristic functions as in \cite{Fristedt}, but the computations are more or less equivalent.} Beginning with the numerator, we rewrite\footnote{For ease of notation, we drop the dependence of $\bm{x}_n$ on $n$ until the very end of the proof.}
\begin{equation}\label{E:QNumerator1}
Q_{q_n}(\dim=n | \bm{X}_{I_n}=\bm{x})
=\frac{Q_{q_n}(\dim=n, \bm{X}_{I_n}=\bm{x})}{Q_{q_n}(\bm{X}_{I_n}=\bm{x})}
=\frac{Q_{q_n}(\dim=n,  \bm{X}_{I_n}=\bm{x})}{\prod_{\k \in I_n} q_n^{a(\k)x_{\k}}\left(1-q_n^{a(\k)}\right)},
\end{equation}
where in the last line we use Proposition \ref{P:BoltzmannSimpleProperties} (2).  Continuing with the numerator, we have
\begin{align}
Q_{q_n}(\dim=n,  \bm{X}_{I_n}=\bm{x})&=q_n^n\prod_{\k \in \N^r} \left(1-q_n^{a(\k)}\right) \#\left\{\rho : \dim(\rho)=n \mbox{ and } X_{\k}(\rho)=x_{\k} \ \text{for $\k \in I_n$}\right\}  \nonumber \\
&=q_n^n\prod_{\k \in \N^r} \left(1-q_n^{a(\k)}\right) \operatorname{coeff}_{\left[w^{n-x_\Ssm}\right]} \prod_{\k \in \N^r \setminus I_n} \left(1-w^{a(\k)}\right)^{-1} \label{E:QNumerator2},
\end{align}
where here and throughout we write $x_\Ssm :=\sum_{\k \in I_n} a(\k)x_{\k}$. Using \eqref{E:QNumerator1} and \eqref{E:QNumerator2}, a short calculation using Cauchy's integral formula with a circle of radius $q_n$ gives
\begin{equation}\label{eq:Qnumerator_integral_exp}
Q_{q_n}(\dim=n| \bm{X}_{I_n}=\bm{x})=\int_{-\frac{1}{2}}^{\frac{1}{2}} \exp\!\left(F_n(2\pi i \theta)\right) d\theta,
\end{equation}
where
\begin{equation*}
F_n(z):=-(n-x_\Ssm)z-\sum_{\k \in \N^r \setminus I_n}
\operatorname{Log}\left( \frac{1-q_n^{a(\k)}e^{a(\k)z}}{1-q_n^{a(\k)}} \right).
\end{equation*}
Note that $F_n (0) = 0$ and we have
\begin{equation*}
F'_n(0)
= x_\Ssm-\sum_{\k \in I_n} \frac{a(\k)q_n^{a(\k)}}{1-q_n^{a(\k)}}
= O(c_n)
\ \mbox{ and } \
F''_n(0) = \sigma_n^2-\mathrm{Var}_{q_n}\left(\sum_{\k \in I_n} a(\k)X_{\k}\right)
\sim \sigma_n^2 ,
\end{equation*}
where we use that $\bm{x}\in B_n$ for the first estimate and equation \eqref{E:BoldXVarSmall2} for the second.
We also define
\begin{equation*}
G_n (\th) := F_n(2\pi i \theta)-F'_n(0)2\pi i \theta - F''_n(0)\frac{(2\pi i \theta)^2}{2}
\end{equation*}
so that
\begin{equation*}
G_n (\th) =
\sum_{\k \in \N^r \setminus I_n} \lp
\operatorname{Log}\!\lp \frac{1-q_n^{a(\k)}}{1-q_n^{a(\k)}e^{2 \pi i \th a(\k)}} \rp
-  \frac{2 \pi i \th a(\k) q_n^{a(\k)}}{1-q_n^{a(\k)}}
+ \frac{(2 \pi \th a(\k))^2 q_n^{a(\k)}}{2 \left(1-q_n^{a(\k)}\right)^2}
\rp.
\end{equation*}
Then by \cite[Lemma 2.2]{BB} and Proposition \ref{P:SaddlePointAsympGrowth}, for $n\in\N$ and $\th \in \mathbb{R}$ we have
\begin{equation}\label{E:TaylorApprox}
\left|G_n(\th) \right| \ll \sum_{\k \in \N^r \setminus I_n} \frac{a(\k)^3 q_n^{a(\k)}}{\left(1-q_n^{a(\k)}\right)^3}|\theta|^3
\ll
n^{\frac{3r+5}{r+3}}|\theta|^3 .
\end{equation}
We want to use the above Taylor approximation on a major arc $|\theta|\leq \th_{0,n}$ with $\th_{0,n} := \varepsilon n^{-\frac{r+1}{r+3}}$ for appropriately small and fixed $0 <\varepsilon \leq \frac{1}{32}$ that we determine below.
We start with
\begin{equation}\label{eq:circle_method_major_arc}
\int_{-\th_{0,n}}^{\th_{0,n}} \exp\left(F_n(2\pi i \theta)\right) d\theta
=
\frac{1}{\b_n}
\int_{-\b_n \th_{0,n}}^{\b_n \th_{0,n}}
e^{- \pi \th^2 + 2 \pi i \eta_n \th + G_n \lp \frac{\th}{\b_n} \rp} d \th,
\end{equation}
where
\begin{equation*}
\eta_n := \frac{F'_n(0)}{\sqrt{2 \pi F''_n (0)}} = O \!\lp \frac{c_n}{\s_n} \rp
\quad \mbox{and} \quad
\b_n := \sqrt{2 \pi F''_n (0)} \sim \sqrt{2 \pi \s_n^2}.
\end{equation*}
We claim that the integral on the right-hand side of \eqref{eq:circle_method_major_arc} tends to one uniformly in $\bm{x}$.
Note that the $\bm{x}$-dependence here is only through the parameter $\eta_n$ (via $x_\Ssm$). To prove this, we first note that
for $\th \in [-\b_n \th_{0,n},\b_n \th_{0,n}]$ we can use the bound \eqref{E:TaylorApprox}
to find (for some constant $C>0$)
\begin{equation*}
\left| \exp \!\lp G_n \!\lp \frac{\th}{\b_n} \rp  \rp  \right|
\leq \exp \!\lp C n^{\frac{3r+5}{r+3}} \frac{|\th|^3}{\b_n^3}  \rp
\leq \exp \!\lp C n^{\frac{3r+5}{r+3}} \frac{\th_{0,n}}{\b_n^2} \th^2 \rp  .
\end{equation*}
According to Proposition \ref{P:SaddlePointAsympGrowth} we have
\begin{equation*}
C n^{\frac{3r+5}{r+3}} \frac{\th_{0,n}}{\b_n^2}
\sim \frac{C \varepsilon }{2 \pi \mathcal{D}_{r,\mathrm{var}}}
\mbox{ as } n \to \infty
\end{equation*}
and hence we can choose and fix $\varepsilon$ sufficiently small to satisfy
\begin{equation}\label{eq:circle_method_major_arc_Gn_bound}
\left| \exp \!\lp G_n \!\lp \frac{\th}{\b_n} \rp  \rp  \right|
\leq e^{\frac{\pi \th^2}{2}}
\quad \mbox{for } \th \in [-\b_n \th_{0,n},\b_n \th_{0,n}] .
\end{equation}
At this point we define \smash{$\g_n := \sqrt{\frac{\s_n}{c_n}}$}, which satisfies $\lim_{n \to \infty} \g_n = \infty$ by \eqref{E:BoldXVarSmall2}. Moreover, $\g_n \leq \b_n \th_{0,n}$ for $n$ sufficiently large,
because $\g_n \leq n^{\frac{1}{6(r+3)}}$ by \eqref{E:BoldXVarSmall2} and
$\b_n \th_{0,n} \asymp n^{\frac{1}{r+3}}$ by
Proposition~\ref{P:SaddlePointAsympGrowth}.
Restricting our attention to such $n$, we use the inequality \eqref{eq:circle_method_major_arc_Gn_bound} and bound
\begin{equation*}
\left| \int_{\g_n \leq |\th| \leq \b_n \th_{0,n}}
e^{- \pi \th^2 + 2 \pi i \eta_n \th + G_n \lp \frac{\th}{\b_n} \rp} d \th \right|
\leq 2 \int_{\g_n}^\infty e^{-\frac{\pi \th^2}{2}}  d \th .
\end{equation*}
The right-hand side (which is independent of $\bm{x}$) tends to zero because $\g_n\to\infty$. We also note that for $|\th| \leq \g_n$ we have
\begin{equation*}
|\eta_n \th| = O\!\lp \sqrt{\frac{c_n}{\s_n}} \rp = o(1)
\quad \mbox{so that }
\left|e^{2 \pi i \eta_n \th} - 1 \right| \ll \sqrt{\frac{c_n}{\s_n}} .
\end{equation*}
This yields the bound, again with the inequality \eqref{eq:circle_method_major_arc_Gn_bound},
\begin{equation*}
\left| \int_{|\th| \leq \g_n}
e^{- \pi \th^2 + G_n \lp \frac{\th}{\b_n} \rp} \lp e^{2\pi i \eta_n \th} - 1 \rp d \th \right|
\ll \sqrt{\frac{c_n}{\s_n}} \int_{\IR} e^{-\frac{\pi \th^2}{2}}  d \th = o(1).
\end{equation*}
Therefore, we have
\begin{equation*}
\int_{-\b_n \th_{0,n}}^{\b_n \th_{0,n}}
e^{- \pi \th^2 + 2 \pi i \eta_n \th + G_n \lp \frac{\th}{\b_n} \rp} d \th
=
\int_{-\g_n}^{\g_n}
e^{- \pi \th^2 + G_n \lp \frac{\th}{\b_n} \rp} d \th + o(1),
\end{equation*}
with the part dropped tending to zero uniformly in $\bm{x} \in B_n$.
For the remaining ($\bm{x}$-independent) integral, note that for $\th \in [-\g_n, \g_n]$ we have
\begin{equation*}
\left| G_n \lp \frac{\th}{\b_n} \rp \right|
\leq C n^{\frac{3r+5}{r+3}} \frac{\g_n^3}{\b_n^3} \ll
n^{\frac{3r+5}{r+3}} \frac{n^{\frac{1}{2r+6}}}{n^{\frac{3r+6}{r+3}}}
= n^{-\frac{1}{2r+6}}.
\end{equation*}
Since $\g_n$ tends to infinity, this then yields the pointwise limit, for $\theta \in \mathbb{R}$,
\begin{equation*}
\lim_{n \to \infty} \lp
e^{- \pi \th^2 + G_n \lp \frac{\th}{\b_n} \rp}
\chi_{[-\g_n,\g_n]} (\th)  \rp
= e^{- \pi \th^2}.
\end{equation*}
Since the integrand is also bounded by $e^{- \frac{\pi\th^2}{2} }$ thanks to \eqref{eq:circle_method_major_arc_Gn_bound}, we can use the Dominated Convergence Theorem
and find that the major arc contribution \eqref{eq:circle_method_major_arc} yields, uniformly in $\bm{x} \in B_n$,
\begin{equation}\label{eq:major_arc_contribution}
\lim_{n \to \infty} \lp \sqrt{2 \pi \s_n^2}
\int_{-\th_{0,n}}^{\th_{0,n}} \exp\left(F_n(2\pi i \theta)\right) d\theta  \rp  = 1.
\end{equation}

Next we discuss the contributions from the minor arc $\th_{0,n} \leq |\th| \leq \frac{1}{2}$. We start with the expression
\begin{equation*}
\left|\exp\left( F_n(2\pi i \theta)\right)\right| =\exp\left(- 2 \sum_{\substack{\k \in \N^r \setminus I_n \\ m \geq 1}}
\frac{q_n^{a(\k)m}}{m} \sin^2 \left(\pi a(\k)m\theta\right)  \right) ,
\end{equation*}
where the right-hand side is $\bm{x}$-independent.
First we show that $\# I_n$ is small in a certain sense, so that we may remove the exclusion of $I_n$ in the sum without paying a large price in appropriate windows for $\k$.
Let $N:= \lceil n^{\frac{2}{r(r+3)}} \rceil$ and let $\Lambda_{N,r}$ be as in Proposition \ref{prop:sinsquare_sum_lower_bound_general_r_case}.
For $\k \in \Lambda_{N,r}$, we have $k_j \asymp n^{\frac{2}{r(r+3)}} \asymp s_n^{-1}$ by Proposition \ref{P:SaddlePointAsympGrowth} so that
$a(s_n \k) \asymp 1$ in this range. Thus,
\begin{equation*}
\mathrm{Var}_{q_n}\!\!\left(\sum_{\k\in I_n} a(\k)X_{\k}\right)\!
\geq
\sum_{\k\in I_n \cap \Lambda_{N,r}} \frac{a(\k)^2q_n^{a(\k)}}{\left(1-q_n^{a(\k)}\right)^{\!2}}
\asymp s_n^{-r(r+1)} \#(I_n \cap \Lambda_{N,r}).
\end{equation*}
Then thanks to the condition \eqref{E:BoldXVarSmall} that we assume on $I_n$
and Proposition \ref{P:SaddlePointAsympGrowth}
we find
\begin{equation*}
\#(I_n \cap \Lambda_{N,r}) =o\!\left(n^{\frac{2}{r+3}}\right) .
\end{equation*}
This leads to the bound
\begin{equation*}
\sum_{\substack{\k \in \N^r \setminus I_n \\ m \geq 1}}
\frac{q_n^{a(\k)m}}{m} \sin^2 \left(\pi a(\k)m\theta\right)
\geq \!\!
\sum_{\k\in \Lambda_{N,r} \setminus I_n} \!\!\!\!
q_n^{a(\k)} \sin^2 \left(\pi a(\k)\theta\right)
=
\sum_{\k\in \Lambda_{N,r}} \!\! q_n^{a(\k)} \sin^2 \left(\pi a(\k)\theta\right) +
o\!\left(n^{\frac{2}{r+3}}\right) .
\end{equation*}
Using that $a(s_n \k) \asymp 1$ on $\LL_{N,r}$, we have
\begin{equation*}
\sum_{\k\in \Lambda_{N,r}} q_n^{a(\k)} \sin^2 \left(\pi a(\k)\theta\right)
\gg
\sum_{\k \in \Lambda_{N,r}} \sin^2 \left(\pi a(\k)\theta\right).
\end{equation*}
Since $\th_{0,n} = \ve n^{-\frac{r+1}{r+3}} \geq \ve N^{-\frac{r(r+1)}{2}}$, we can then use Corollary \ref{cor:sinsquare_sum_lower_bound} for $n$ sufficiently large to find
\begin{equation*}
\sum_{\k \in \Lambda_{N,r}} q_n^{a(\k)} \sin^2 \left(\pi a(\k)\theta\right)
\gg n^{\frac{2}{r+3}}.
\end{equation*}
So for $n$ sufficiently large we find (with some constant $D>0$ depending on $r$ and $\varepsilon$)
\begin{equation}\label{eq:minor_arc_contribution}
\left| \int_{\th_{0,n} \leq |\th| \leq \frac{1}{2}} \exp\left(F_n(2\pi i \theta)\right) d\theta \right|
\leq
\exp \!\lp -D n^{\frac{2}{r+3}} \rp
\end{equation}
and the minor arc contribution is exponentially small compared to the major arc contribution.

Inserting the major arc \eqref{eq:major_arc_contribution} and minor arc \eqref{eq:minor_arc_contribution} contributions in equation \eqref{eq:Qnumerator_integral_exp} we obtain
\begin{equation*}
\lim_{n \to \infty} \lp \sqrt{2 \pi \s_n^2} \,
Q_{q_n}(\dim=n| \bm{X}_{I_n}=\bm{x}_n)  \rp  = 1
\end{equation*}
uniformly in $\bm{x}_n \in B_n$.
Employing this result together with the case $I_n = \emptyset$ implies
the property \eqref{E:QuotientUniformto1}, which then finishes the proof of equation \eqref{eq:total_variation_pn_qn}. The case $I_n = \emptyset$ also establishes \eqref{eq:Qqn_prob_dimension}.
\end{proof}

\begin{rem}
Let $q_n$ be as in Proposition \ref{P:SaddlePointexist} and
$\mathcal{D}_{r,var}$ as in Proposition~\ref{P:SaddlePointAsympGrowth}.
Using Proposition \ref{P:EquivalenceofEnsembles} and recalling the definition of $Q_{q_n}(\dim =n)$, one obtains the following asymptotic for $p_r(n)$, the number of $n$-dimensional $\mathfrak{sl}_{r+1}(\mathbb{C})$-representations, in terms of $q_n$:
$$
p_r(n) \sim \frac{n^{-\frac{r+2}{r+3}}
}{\sqrt{2\pi \mathcal{D}_{r,\mathrm{var}}}}
q_n^{-n}
\prod_{\k \in \N^r} \left(1-q_n^{a(\k)}\right)^{-1}.
$$
If $q_n$ is computed to enough accuracy, then the above may be improved to an asymptotic formula that generalizes Romik's asymptotic formula for $r=2$ in \cite{Romik}.  As this would likely require a more careful study of the Witten zeta functions or more terms in the expansion in Proposition \ref{prop:irrep_partial_sum_asymptotics}, we leave this as an open problem.
\end{rem}

The following corollary to Proposition \ref{P:EquivalenceofEnsembles} gives simple criteria for equivalence of ensembles of the random variables we consider in subsequent sections.

\begin{cor}\label{C:EquivalenceofEnsembles}
Let $\bm{X}_{I_n}:=\left(X_{\bm{k}}\right)_{\k \in I_n}$, where we allow the index set $I_n$ to change with $n$.  Suppose that at least one of the following four conditions hold:
\begin{enumerate}[leftmargin=*,label=\rm(\arabic*)]
\item $I_n=\prod_{j=1}^r [k^{[1]}_j,k^{[2]}_j]$, where $s_nk_j^{[2]} \to 0$ for some $j$,
\item $I_n=\prod_{j=1}^r [k^{[1]}_j,k^{[2]}_j]$, where $s_nk_j^{[1]} \to \infty$ for some $j$,
\item $\sup_{\k \in I_n} a(s_n\k)=o(1)$ as $n \to \infty$,
\item $\inf_{\k \in I_n} a(s_n\k) \to \infty$ as $n \to \infty$.
\end{enumerate}
Then we have
$$
\lim_{n \to \infty} d_{\operatorname{TV}}\!\left(Q_{q_n} \!\left(\bm{X}_{I_n}^{-1}\right), P_n \!\left(\bm{X}^{-1}_{I_n}\right)\right) =0.
$$
\end{cor}
\begin{proof}
Using Proposition \ref{P:SaddlePointAsympGrowth} and the fact that $a(\bm{t})$ has positive coefficients, we have
$$
\s_n^{-2} \, \mathrm{Var}_{q_n}\!\!\left(\sum_{\k \in I_n} a(\k) X_{\k}\right)
\asymp
\sum_{\k \in I_n} \frac{a(s_n\k)^2q_n^{a(\k)}}{\left(1-q_n^{a(\k)}\right)^2}s_n^r \leq \int_{[0,\infty)^r} \frac{a(\bm{t})^2e^{-a(\bm{t})}}{\left(1-e^{-a(\bm{t})}\right)^2} \chi_{a(\bm{t}) \leq \sup_{\k\in I_n} \! a\left(s_n\k\right) \,} \bm{dt}.
$$
If (3) holds, then the integrand tends pointwise to zero.
The integrand is dominated by the same function without the indicator function, which is integrable by Proposition \ref{P:SaddlePointAsympGrowth}. So the integral tends to zero by the Dominated Convergence Theorem and the conclusion of the corollary follows from Proposition \ref{P:EquivalenceofEnsembles}.
The proofs for the other cases are analogous. 
\end{proof}

In practice, statistics of interest may depend on subsets $I_n$ that do not satisfy the above condition, and so Corollary \ref{C:EquivalenceofEnsembles} may not be readily applied.  In fact, this is true for several random variables in Table \ref{Table:slrstats}. In such cases, we use Corollary \ref{C:EquivalenceofEnsembles} and the following ``rare events lemma" in tandem.

\begin{lem}[Rare events lemma] \label{lem:ExpSmallPrinciple}
Suppose that an event $A_n$ is such that $s_n^{-\frac{r(r+2)}{2}} Q_{q_n}(A_n) = o(1)$ as $n \to \infty$. Then we have $\lim_{n \to \infty} P_n(A_n)=0$.
\end{lem}
\begin{rem}
Clearly the conclusion of the lemma follows in the case that the event $A_n$ has exponentially small probability under the Boltzmann model - i.e.,~$Q_{q_n}(A_n) \leq e^{-n^{\delta}}$ for some $\delta>0$ and $n$ sufficiently large.
\end{rem}
\begin{proof}[Proof of Lemma \ref{lem:ExpSmallPrinciple}]
We have
\begin{equation*}
P_n(A_n) =Q_{q_n}(A_n | \dim=n) =\frac{Q_{q_n}(A_n, \dim=n)}{Q_{q_n}(\dim=n)} \leq \frac{Q_{q_n}(A_n)}{Q_{q_n}(\dim=n)} \asymp s_n^{-\frac{r(r+2)}{2}} Q_{q_n}(A_n) ,
\end{equation*}
where the last part follows from Proposition \ref{P:SaddlePointAsympGrowth} and Proposition \ref{P:EquivalenceofEnsembles}.
\end{proof}

\section{Multiplicities of small irreducible representations}\label{S:SmallMultiplicities}
We start our discussion of the results from Table \ref{Table:slrstats} with the distribution of the multiplicities of irreducible representations of small dimension.  As in the case of partitions studied in \cite[Theorems 2.1 and 2.2]{Fristedt}, note that the joint distribution holds in a shorter range, compared to the distribution of a single multiplicity.

\begin{thm}\hspace{0cm} \label{T:Mult}
	\begin{enumerate}[leftmargin=*,label=\rm(\arabic*)]
		\item Suppose that $\bm{k}_n \!\in\! \mathbb{N}^r$ may depend on $n$ as long as $a(s_n\bm{k}_n)=o(1)$.  Then for $x \in [0,\infty)$,
		\begin{equation*}
		\lim_{n \to \infty} P_n\!\left(a(s_n\k_n)X_{\k_n}\leq x \right) = 1-e^{-x}.
		\end{equation*}

		\item Let $\bm{\kappa}_n := (\kappa_{n,1}, \ldots, \kappa_{n,r}) \in \mathbb{N}^r$ be a sequence of vectors with nondecreasing components
		such that $a(s_n\bm{\kappa}_n)= o(s_n^{\frac{r(r+1)}{r+3}})$.
		Then for any set of $x_{\k} \in [0,\infty)$ with $\k \in \prod_{j=1}^r \!\left[ 1,\sup_n (\kappa_{n,j}) \right]$,
		\begin{equation*}
		\lim_{n \to \infty}
		P_n\!\left(a(s_n\k)X_{\k}\leq x_{\k} \mbox{ for all }
		\k \in \prod_{j=1}^r\left[1,\kappa_{n,j} \right] \right)
		=
		\lim_{n \to\infty}
		\prod_{k_1 = 1}^{\kappa_{n,1}} \cdots \prod_{k_r = 1}^{\kappa_{n,r}}
		\!\left(1-e^{-x_{\k}}\right).
		\end{equation*}
	\end{enumerate}
\end{thm}
\begin{proof}
	\begin{enumerate}[wide,labelwidth=!,labelindent=0pt,label=\rm(\arabic*)]
		\item The claim follows immediately from Corollary \ref{C:EquivalenceofEnsembles} (3) and Proposition \ref{P:BoltzmannSimpleProperties} (2).

		\item Note that $s_n \to 0$, so Corollary \ref{C:EquivalenceofEnsembles} (3) applies and we have (with $I_n := \prod_{j=1}^r \left[1,\kappa_{n,j} \right]$)
		\begin{equation}\label{eq:small_irrep_joint_Pn_Qqn_difference_limit}
		\lim_{n \to \infty} \lp
		P_n\!\left(a(s_n\k)X_{\k}\leq x_{\k} \mbox{ for all }
		\k \in I_n \right)
		-
		Q_{q_n}\!\left(a(s_n\k)X_{\k}\leq x_{\k} \mbox{ for all }
		\k \in I_n \right)
		\rp = 0.
		\end{equation}
		Using the independence and distributions of $X_{\bm{k}}$ from Proposition \ref{P:BoltzmannSimpleProperties} (2), we have
		\begin{equation*}
		Q_{q_n}\!\left(a(s_n\k)X_{\k}\leq x_{\k} \mbox{ for all } \k \in I_n \right)
		= \prod_{\bm{k} \in I_n}
		\left(1-e^{-x_{\k}} + \ve_n (\bm{k}) \right),
		\end{equation*}
		where the error terms satisfy $0 \leq \ve_n (\bm{k}) \leq a(s_n\k)$.
		We now use that, for any $0 \leq a_j \leq 1$ and $0 \leq b_j$,
		\begin{equation*}
		\prod_{j=1}^m a_j \leq
		\prod_{j=1}^m (a_j + b_j) \leq \prod_{j=1}^m a_j + \prod_{j=1}^m (1+b_j) -1
		\leq \prod_{j=1}^m a_j + e^{\sum_{j=1}^m b_j} - 1
		\end{equation*}
		so that
		\begin{equation*}
		Q_{q_n}\!\left(a(s_n\k)X_{\k}\leq x_{\k} \mbox{ for all } \k \in I_n \right)
		= \prod_{\bm{k} \in I_n} \left(1-e^{-x_{\k}} \right)
		+ O\!\lp \exp\!\left(\sum_{\k \in I_n} a(s_n\k) \right)-1 \rp .
		\end{equation*}
		For the sum in the exponential we have
		\begin{equation*}
		\sum_{\k \in I_n} a(s_n\k)
		\leq s_n^{\frac{r(r+1)}{2}}\sum_{m =1}^{a(\bm{\kappa}_n)} m \varrho_r(m)
		=
		O\!\lp s_n^{\frac{r(r+1)}{2}} a(\bm{\kappa}_n)^{1+\frac{2}{r+1}} \rp,
		\end{equation*}
		where the final step follows from Abel partial summation and
		Proposition \ref{prop:irrep_partial_sum_asymptotics}.
		This result is then $o(1)$ by the assumption on $a(s_n\bm{\kappa}_n)$. \
		Thus we have
		\begin{equation*}
		\lim_{n \to \infty}
		Q_{q_n}\!\left(a(s_n\k)X_{\k}\leq x_{\k} \mbox{ for all } \k \in I_n \right)
		=
		\lim_{n \to \infty}
		\prod_{k_1 = 1}^{\kappa_{n,1}} \cdots \prod_{k_r = 1}^{\kappa_{n,r}}
		\!\left(1-e^{-x_{\k}}\right)
		\end{equation*}
		with the limit on the right hand side converging thanks to the assumption that the components of $\bm{\kappa}_n$ are nondecreasing in $n$. Using this result with equation
\eqref{eq:small_irrep_joint_Pn_Qqn_difference_limit}, the theorem statement follows.
\qedhere
\end{enumerate}
\end{proof}

\section{Maximum dimension and height}\label{S:MaxDimHeight}

\subsection{Maximum dimension}
Next, we prove the distribution for the largest dimension among all irreducible representations occurring in the decomposition.  Recall that this is $D=\max_{X_{\k}>0} (a(\k))$.

\begin{thm}\label{T:lgstajk}
For any $x \in \IR$ we have
$$
\lim_{n \to \infty} P_n\!\left(s_n^{\frac{r(r+1)}{2}}D-\log\left(s_n^{-r}\right)+\frac{r-1}{r+1}\log\left(\log\left(s_n^{-r}\right)\right) -\log\left(\frac{2C_r}{r+1}\right)\leq x\right) = e^{-e^{-x}}.
$$
\end{thm}
\begin{proof}
The theorem statement can be written as
$\lim_{n \to \infty} P_n\left(D \leq\ell_{x,n}\right)=e^{-e^{-x}}$, where
\begin{equation*}
\ell_{x,n}:=
s_n^{-\frac{r(r+1)}{2}}
\left(
\w_n-\frac{r-1}{r+1}\log (\w_n) +\log\left(\frac{2C_r}{r+1}\right)+x \right)
\ \mbox{ and } \
\w_n := \log\!\left(s_n^{-r}\right).
\end{equation*}
From now on we restrict our attention to $n$ sufficiently large, where $\w_n >0$ and $\ell_{x,n} > 0$.
First note that the event $D \leq\ell_{x,n}$ can be equivalently written as
\begin{equation*}
X_{\bm{k}} = 0 \quad \mbox{for all }
\bm{k} \in I_n := \{\bm{k} \in \IN^r : \ a(\bm{k}) > \ell_{x,n} \}
\end{equation*}
and can be studied with the marginal distribution for the random variable $\bm{X}_{I_n}$.
Since we have
\begin{equation}\label{E:ellxnlarge}
\lim_{n\to \infty} \lp s_n^{\frac{r(r+1)}{2}} \ell_{x,n} \rp = \infty,
\end{equation}
the set $I_n$ satisfies Corollary \ref{C:EquivalenceofEnsembles} (4) and it is enough to prove the claimed limit for $Q_{q_n}$ as\footnote{
In fact, following our arguments, one can prove a stronger result that allows $x$ to depend on $n$ as well and show
\begin{equation*}
-\log\!\left(Q_{q_n}(D\leq \ell_{x_n,n})\right) = e^{-x_n}(1+o(1))+o(1)
\mbox{ as } n \to \infty
\end{equation*}
for any sequence of real numbers $x_n$ such that
$- \w_n^{\frac{1}{2}} \leq x_n \leq \w_n^{\frac{1}{2}}$
for all $n$ and that the convergence is uniform in the sequence $x_n$.
As in \cite{Fristedt}, such results are required to go beyond the weak convergence that we establish here.
}
\begin{equation}\label{E:MaxDimLimitQn}
\lim_{n \to \infty} Q_{q_n} \!\left(D \leq\ell_{x,n}\right)=e^{-e^{-x}} .
\end{equation}

To begin our proof, we recall the definition of $D$ to find
\begin{equation*}
\sum_{\rho: \, D\leq \ell_{x,n}} q_n^{\dim(\rho)}=\prod_{\substack{\k \in \mathbb{N}^r \\ a(\k) \leq \ell_{x,n}}} \frac{1}{1-q_n^{a(\k)}}
\end{equation*}
so that
\begin{equation*}
-\log \!\left(Q_{q_n}(D\leq \ell_{x,n})\right)=-\sum_{\substack{\k \in \mathbb{N}^r \\ a(\k) > \ell_{x,n}}} \log\left(1-q_n^{a(\k)}\right).
\end{equation*}
Note that for $a(\k)> \ell_{x,n}$ we have
\begin{equation*}
q_n^{a(\k)} < \exp\!\lp - s_n^{\frac{r(r+1)}{2}} \ell_{x,n}  \rp.
\end{equation*}
Since $t \leq -\log(1-t) \leq t+t^2$ for $t\in [0,\frac{1}{2}]$, by \eqref{E:ellxnlarge} we use it for $n$ sufficiently large to get
\begin{equation*}
-\log(Q_{q_n}(D\leq \ell_{x,n}))
=
\sum_{m > \ell_{x,n}}\! q_n^{m}\varrho_r(m) +
O\!\left( \sum_{ m > \ell_{x,n}} q_n^{2m}\varrho_r(m)\right)\!.
\end{equation*}
Abel partial summation gives, after some simplification,
$$
\sum_{m > \ell_{x,n}} q_n^{m}\varrho_r(m)=-e^{-s_n^{\frac{r(r+1)}{2}} \ell_{x,n}}R_r(\ell_{x,n})
+\int_{\ell_{x,n} s_n^{\frac{r(r+1)}{2}}}^{\infty} R_r\!\left(s_n^{-\frac{r(r+1)}{2}} t \right) e^{-t} dt.
$$
Next, Proposition \ref{prop:irrep_partial_sum_asymptotics} yields
\begin{align}
\sum_{m > \ell_{x,n}} q_n^{m}\varrho_r(m)
&=
-C_re^{-s_n^{\frac{r(r+1)}{2}} \ell_{x,n}} \ell_{x,n}^{\frac{2}{r+1}}
+C_r s_n^{-r} \Gamma\!\left(\frac{r+3}{r+1}, \ell_{x,n} s_n^{\frac{r(r+1)}{2}} \right)\nonumber \\*
& \hspace{8mm} + O\left(e^{-s_n^{\frac{r(r+1)}{2}} \ell_{x,n}}
\ell_{x,n}^{\frac{2(r-1)}{r^2}}
+ s_n^{-\frac{r^2-1}{r}}  \Gamma\!\left(\frac{r^2+2r-2}{r^2}, \ell_{x,n} s_n^{\frac{r(r+1)}{2}} \right)\right), \label{E:Dmainterm}
\end{align}
where
\begin{equation}\label{eq:incomplete_gamma_definition}
\Gamma(a,v) := \int_v^\infty t^{a-1} e^{-t} dt
\end{equation}
denotes the {\it incomplete $\Gamma$-function}. From \cite[(8.11.2)]{NIST}, we have, as $v \to \infty$,
\begin{equation}\label{eq:incomplete_gamma_asymptotic_behavior}
	\Gamma (a,v) = e^{-v}\left(v^{a-1}+(a-1)v^{a-2}+O\!\left(v^{a-3}\right)\right).
\end{equation}
so in particular,
$$
\Gamma\!\left(\frac{r+3}{r+1},v\right) = e^{-v}\left(v^{\frac{2}{r+1}}+\frac{2}{r+1}v^{-\frac{r-1}{r+1}}+O\!\left(v^{-\frac{2r}{r+1}}\right)\right), \qquad
\Gamma\!\left(\frac{r^2+2r-2}{r^2},v\right) \sim v^{\frac{2r-2}{r^2}}e^{-v} .
$$
By \eqref{E:ellxnlarge}, we may use the above expansions in \eqref{E:Dmainterm}.
The term on the far left in \eqref{E:Dmainterm} cancels with the leading term from the second contribution, and we get, after some simplification,
$$
\sum_{m > \ell_{x,n}} q_n^{m}\varrho_r(m)
=
\frac{2C_r}{r+1} s_n^{-\frac{r(r+1)}{2}} \ell_{x,n}^{-\frac{r-1}{r+1}} e^{-s_n^{\frac{r(r+1)}{2}} \ell_{x,n}}
(1+o(1)) + O\!\left(e^{-s_n^{\frac{r(r+1)}{2}} \ell_{x,n}}  \ell_{x,n}^{\frac{2(r-1)}{r^2}}
\right).
$$
The error term may be bounded as
\begin{equation*}
e^{-s_n^{\frac{r(r+1)}{2}} \ell_{x,n}}  \ell_{x,n}^{\frac{2(r-1)}{r^2}}
\ll
e^{- \w_n (1+o(1))} s_n^{-\frac{r^2-1}{r}} \w_n^{\frac{2(r-1)}{r^2}}
=
e^{- \w_n \lp \frac{1}{r^2}+o(1) \rp}  \w_n^{\frac{2(r-1)}{r^2}},
\end{equation*}
which is $o(1)$ because $\w_n$ tends to infinity as $n \to \infty$.
The leading term  may be simplified to
\begin{equation*}
\frac{2C_r}{r+1} s_n^{-\frac{r(r+1)}{2}} \ell_{x,n}^{-\frac{r-1}{r+1}} e^{-s_n^{\frac{r(r+1)}{2}} \ell_{x,n}}
=
e^{-x}
\left(1 + \frac{1}{\w_n} \lp x + \log\left(\frac{2C_r}{r+1}\right)
- \frac{r-1}{r+1} \log (\w_n) \rp \right)^{-\frac{r-1}{r+1}}
\end{equation*}
and it contributes as $e^{-x}(1+o(1))$.
So overall we have
\begin{equation*}
\sum_{m > \ell_{x,n}} q_n^{m}\varrho_r(m) = e^{-x}(1+o(1))+o(1).
\end{equation*}
A similar calculation shows that
$$
\sum_{m > \ell_{x,n}} q_n^{2m}\varrho_r(m)  = o(1),
$$
thereby completing the proof of \eqref{E:MaxDimLimitQn}.
\end{proof}

\subsection{Height} \label{S:Heightdistribution}
We continue our discussion with the distribution of height, where we recall that the height of a representation is given as
$H=\max_{X_{\k}>0} (L(\bm{k}-\bm{1}))$ with $L$ defined in \eqref{eq:height_l_function_defn}.
Our main line of attack to develop a limiting probability under $P_n$ is once again proving such a limit for $Q_{q_n}$ and then transferring this to $P_n$ with the equivalence of ensembles.
In our discussion below, it becomes clear that irreducible representations with $\bm{k}$ of the form $(k,1,\ldots,1)$ or $(1,\ldots,1,k)$ dominate the asymptotic behavior under $Q_{q_n}$. So we start with a few technical lemmas that we employ in establishing this behavior.
The first of these is an easy consequence of the definition of $a(\bm{k})$ in \eqref{E:akDef} and of $L(\bm{k})$ in \eqref{eq:height_l_function_defn}.

\begin{lem}\label{lem:height_dim_representation_bounds}
Let $\bm{k} \in \IN^r$.
\begin{enumerate}[label=\rm(\arabic*),leftmargin=*]
\item Defining $k_{\mathrm{max}} := \max\{\bm{k}\}$, we have
\begin{equation*}
\frac{12}{r(r+1)(r+2)} L(\bm{k}-\bm{1}) \leq k_{\mathrm{max}}-1 \leq \frac{2}{r} L(\bm{k}-\bm{1}) .
\end{equation*}

\item For $j \in \{1,2,\ldots,r\}$, we have $a(\bm{k}) \gg k_j^{j(r+1-j)}$ (with the implied constant dependent on $r$).
\end{enumerate}
\end{lem}

\begin{rem}
If $k_j = \max\{\bm{k}\}$, then Lemma \ref{lem:height_dim_representation_bounds} implies that
\begin{equation}\label{eq:dimension_height_rough_bound}
a(\bm{k}) \gg L(\bm{k}-\bm{1})^{j(r+1-j)} .
\end{equation}
In particular, independent of the position of $\max\{\bm{k}\}$ we overall have
$a(\bm{k}) \gg L(\bm{k}-\bm{1})^{r}$.
\end{rem}

For our subsequent arguments, we need to separate the contributions of vectors $\bm{k}$ of the form
$(k,1,\ldots,1)$ or $(k,1,\ldots,1,2,1,\ldots,1)$ and their reverse. More precisely, we
partition $\IN^r$ to
\begin{equation}\label{eq:Qqn_height_distribution_Lambda_set_definition}
\LL_1 := \IN \times S_1 \cup S_1 \times \IN,
\quad
\LL_2 := \IN_{\geq 2} \times S_2 \cup S_2 \times \IN_{\geq 2},
\andd
\LL := \IN^r \setminus (\LL_1 \cup \LL_2),
\end{equation}
where
\begin{equation*}
S_1 := \{(1,\ldots,1)\}, \ \
S_2 := \{(2,1,\ldots,1),(1,2,1,\ldots,1),\ldots,(1,\ldots,1,2)\} \subset \IN^{r-1}.
\end{equation*}
Here we note that for vectors of the form $\bm{k} = (k,1,\ldots,1)$, we have
$a(\bm{k}) = \frac{k(k+1)\cdots (k+r-1)}{r!}$
and $\frac{2}{r} L(\bm{k}-\bm{1}) = (k-1)$ so that
$a(\bm{k}) \sim \frac{1}{r!} (\frac{2}{r} L(\bm{k}-\bm{1})) ^r$ as $k \to \infty$.
Similarly, for vectors of the form $\bm{k} = (k,2,\ldots,1)$ we have
$a(\bm{k}) \sim \frac{1}{(r-1)!} \lp \frac{2}{r} L(\bm{k}-\bm{1}) \rp^r$ as $k \to \infty$. Our next goal is to establish a more precise lower bound for $a(\bm{k})$ in terms of $L(\bm{k}-\bm{1})$ that improves the constants in these asymptotic results by a factor of $\frac{3r}{2}$ and $\frac{3}{2}$, respectively, for vectors in $\Lambda$.

\begin{lem}\label{lem:height_representation_dominant_weights}
There exists $N_r \in \IN$ such that for all $\bm{k} \in \LL$ with $L(\bm{k}-\bm{1}) > N_r$ we have
\begin{equation*}
a(\bm{k}) \geq \frac{3}{2} \frac{1}{(r-1)!} \lp \frac{2}{r} L(\bm{k}-\bm{1}) \rp^r .
\end{equation*}
\end{lem}
\begin{proof}
If $k_j = \max\{\bm{k}\}$ with $1 < j < r$, then the result follows from
Lemma \ref{lem:height_dim_representation_bounds} and \eqref{eq:dimension_height_rough_bound}.
So recalling the symmetry under $(k_1,\ldots,k_r) \leftrightarrow (k_r,\ldots,k_1)$,
it is enough to prove the result if $\max\{\bm{k}\}$ is attained by $k_1$.

For notational convenience, we shift $\bm{k}$ by $\bm{1}$ and prove that for all $\bm{k} \in \IN_0^r$ satisfying $k_1 = \max\{\bm{k}\}$,
\begin{equation*}
(k_2, \ldots, k_r) \not\in \{(0,\ldots,0),(1,0,\ldots,0),(0,1,0,\ldots,0),
\ldots, (0,\ldots,0,1) \},
\end{equation*}
and $k_1 \gg 1$ we have
\begin{equation}\label{eq:lem_height_representation_dominant_weights_shifted}
I(\bm{k}) := \frac{1}{r} \frac{r!  a(\bm{k}+\bm{1})}{\lp \frac{2}{r} L(\bm{k}) \rp^r} \geq \frac{3}{2}.
\end{equation}
Here it is useful to note that
\begin{equation}\label{eq:lem_height_representation_dimension}
r!  a(\bm{k}+\bm{1}) =
(k_1+1) (k_1+k_2+2) \cdots (k_1+\ldots+k_r+r)
\prod_{2 \leq \ell \leq j \leq r}
\lp 1 + \frac{k_\ell+\ldots+k_j}{j-\ell+1} \rp .
\end{equation}
For the case $r=2$, the inequality \eqref{eq:lem_height_representation_dominant_weights_shifted} is not hard to establish for $k_1,k_2 \geq 2$ with an explicit computation. So we also assume $r \geq 3$ from now on and define 
\begin{equation*}
\kappa_{m} := \max\{k_2,\ldots,k_r\}.
\end{equation*}
We have $1 \leq \kappa_{m} \leq k_1$ according to our assumptions for \eqref{eq:lem_height_representation_dominant_weights_shifted}. With this notation, we bound~$L(\bm{k})$~as
\begin{equation}\label{eq:lem_height_representation_height_bound0}
\frac{2}{r} L(\bm{k}) = \frac{1}{r} \sum_{j=1}^r j (r+1-j) k_j
\leq k_1 + r^2 \kappa_{m}.
\end{equation}
To bound $a(\bm{k}+\bm{1})$, we distinguish two cases depending on whether $\kappa_{m} = 1$ or not.

\noindent\textbf{Case 1:} $\kappa_{m} = 1$ and more than one of $k_2,\ldots,k_r$ are equal to one. Assuming that $k_{n_1} = k_{n_2} = 1$ for some $2 \leq n_1 < n_2 \leq r$ we bound
\begin{align*}
&\!\prod_{2 \leq \ell \leq j \leq r}\!
\!\lp 1 \!+\! \frac{k_\ell+\ldots+k_j}{j-\ell+1} \rp\! \geq
(1+k_{n_1}) \!\prod_{2 \leq \ell \leq n_2}\! \!\lp\! 1 \!+\! \frac{k_\ell+\ldots+k_{n_2}}{n_2-\ell+1} \rp\! \prod_{n_2<j \leq r}\! \!\lp 1 + \frac{k_2+\ldots+k_j}{j-1} \rp\!
\\
&\hspace{4cm}\geq 2
\prod_{2 \leq \ell \leq n_2}  \lp 1 + \frac{1}{n_2-\ell+1} \rp
\prod_{n_2<j \leq r} \lp 1 + \frac{1}{j-1} \rp
=
2 \prod_{j=1}^{r-1} \lp 1 + \frac{1}{j} \rp = 2r .
\end{align*}
With equation \eqref{eq:lem_height_representation_dimension}, we then find
$r!  a(\bm{k}+\bm{1}) \geq  2r k_1^r$, which together with \eqref{eq:lem_height_representation_height_bound0}
shows that
\begin{equation*}
I (\bm{k}) \geq \frac{2 k_1^r}{\lp k_1 + r^2 \rp^r}  .
\end{equation*}
The right-hand side tends to $2$ as $k_1 \to \infty$ and hence $I (\bm{k}) \geq \frac{3}{2}$ for $k_1$ sufficiently large.

\noindent\textbf{Case 2:} $2 \leq \kappa_{m} \leq k_1$.
Supposing that $\kappa_{m}$ is attained at position $n$, we bound
\begin{equation*}
\prod_{2 \leq \ell \leq j \leq r} \!\!\!
\lp 1 + \frac{k_\ell+\ldots+k_j}{j-\ell+1} \rp
\! \geq  \!
\prod_{2 \leq \ell \leq n} \!\!\! \lp 1 + \frac{k_\ell+\ldots+k_n}{n-\ell+1} \rp \prod_{n<j \leq r} \!\!\! \lp 1 + \frac{k_2+\ldots+k_j}{j-1} \rp
\! \geq  \!
\prod_{j=1}^{r-1} \!\lp 1 + \frac{\kappa_{m}}{j} \rp \! .
\end{equation*}
Using this in \eqref{eq:lem_height_representation_dimension} together with the bound
\eqref{eq:lem_height_representation_height_bound0} yields
\begin{equation}\label{eq:lem_height_representation_level_bound_prep2}
I(\bm{k}) \geq  \frac{1}{r}
\frac{k_1^r}{(k_1 + r^2 \kappa_{m})^r}
\prod_{j=1}^{r-1} \lp 1 + \frac{\kappa_{m}}{j} \rp .
\end{equation}
If $2 \leq \kappa_{m} \leq 34$, then we have
\begin{equation*}
\prod_{j=1}^{r-1} \lp 1 + \frac{\kappa_{m}}{j} \rp \geq
\prod_{j=1}^{r-1} \lp 1 + \frac{2}{j} \rp = \frac{r(r+1)}{2}
\quad \mbox{and hence }
I(\bm{k}) \geq
\frac{r+1}{2} \frac{k_1^r}{(k_1 + 34r^2)^r} .
\end{equation*}
Since $\frac{r+1}{2} \geq 2$, the inequality \eqref{eq:lem_height_representation_dominant_weights_shifted} is then satisfied in this case for $k_1$ sufficiently large.

So we finally consider the case $35 \leq \kappa_{m} \leq k_1$.
We start by bounding (with $1 \leq j \leq r$ arbitrary)
\begin{equation*}
k_1 + r^2 \kappa_{m}
\leq
\sqrt{k_1^2 + r^4 j \kappa_{m}}
\sqrt{1 + \frac{\kappa_{m}}{j}}
\leq
\left(k_1 + r^5\right) \sqrt{1 + \frac{\kappa_{m}}{j}},
\end{equation*}
where we use the Cauchy-Schwarz inequality in the first step
with $\kappa_{m} \leq k_1$ and $j \leq r$ in the second step.
We apply this bound on the $r$ factors in the denominator of \eqref{eq:lem_height_representation_level_bound_prep2} by employing it with each of $j \in \{1,2,\ldots, r-2\}$ once and twice with $j = r-1$. This yields
\begin{equation*}
I(\bm{k}) \geq  \frac{k_1^r}{(k_1 + r^5)^r}
\sqrt{\frac{1}{r^2} \prod_{j=1}^{r-2} \lp 1 + \frac{\kappa_{m}}{j} \rp}
\geq
\frac{k_1^r}{(k_1 + r^5)^r}
\sqrt{\frac{1}{r^2} \prod_{j=1}^{r-2} \lp 1 + \frac{35}{j} \rp}.
\end{equation*}
Since the factor in the square root is $\geq 4$ for $r \geq 3$, the inequality \eqref{eq:lem_height_representation_dominant_weights_shifted} is satisfied for $k_1$ sufficiently large in this case as well.
\end{proof}

As our last preparatory step, we give bounds for vectors in $\Lambda_2$ (see \eqref{eq:Qqn_height_distribution_Lambda_set_definition} for the definition).
\begin{lem}\label{lem:height_representation_lambda2_bounds}
For any $\bm{k} := (k,1,\ldots,1,2,1,\ldots,1)$ or $(1,\ldots,1,2,1,\ldots,1,k) \in \LL_2$ we have
\begin{equation*}
a(\bm{k}) \geq \frac{k^r}{(r-1)!}
\andd
\frac{2}{r} L(\bm{k}-\bm{1}) \leq k+r .
\end{equation*}
\end{lem}
\begin{proof}
Let $\bm{k} := (k,1,\ldots,1,2,1,\ldots,1)$ with $k_n = 2$ for some $2 \leq n \leq r$;
the other case follows from the symmetry
$(k_1,\ldots,k_r) \leftrightarrow (k_r,\ldots,k_1)$.
Then we have
\begin{equation*}
\frac{2}{r} L(\bm{k}-\bm{1}) = \frac{1}{r} \sum_{j=1}^r j (r+1-j) (k_j-1)
= \frac{1}{r} \lp r (k-1) + n(r+1-n) \rp \leq k-1 + \frac{(r+1)^2}{4r} ,
\end{equation*}
proving the second claim since $\frac{(r+1)^2}{4r} \leq r$.
The first one follows from \eqref{eq:lem_height_representation_dimension} as in~Lemma~\ref{lem:height_representation_dominant_weights}.
\end{proof}

We are now ready to give the height distribution
with $\a_n = \GG (r+1)  \log ( 2 \GG (r) s_n^{-\frac{r+1}{2}} )$ from \eqref{eq:introduction_C_r_alpha_n_definitions}.

\begin{thm}\label{T:Height}
For any $x \in \IR$ we have
\begin{equation*}
\lim_{n \to \infty} P_n \!\lp
\frac{2}{r!} s_n^{\frac{r+1}{2}} \a_n^{\frac{r-1}{r}} H - \frac{\a_n}{(r-1)!}
+\frac{r-1}{r} \log (\a_n) \leq x \rp = e^{-e^{-x}}.
\end{equation*}
\end{thm}
\begin{proof}
The theorem statement can be rewritten as $\lim_{n \to \infty} P_n(H \leq \ell_{x,n}) = e^{-e^{-x}}$ with
\begin{equation}\label{eq:Qqn_height_ell_x_n_definition}
\ell_{x,n} := \frac{r}{2} s_n^{-\frac{r+1}{2}} \a_n^{\frac{1}{r}}
\lp 1 + \GG(r) \frac{x - \frac{r-1}{r} \log (\a_n)}{\a_n} \rp .
\end{equation}
From now on, we assume that $n$ is sufficiently large so that $\a_n > 0$ and $\ell_{x,n} > 0$.
As in Theorem \ref{T:lgstajk}, we first note that the event $H \leq \ell_{x,n}$ can be equivalently written as
\begin{equation*}
X_{\bm{k}} = 0 \quad \mbox{for all }
\bm{k} \in I_n := \{\bm{k} \in \IN^r : \ L(\bm{k}-\bm{1}) > \ell_{x,n} \}
\end{equation*}
and can be studied with the marginal distribution for the random variable $\bm{X}_{I_n}$.
Now recall that $a(\bm{k}) \gg L(\bm{k}-\bm{1})^{r}$ by our remarks surrounding
\eqref{eq:dimension_height_rough_bound}. So for $\bm{k} \in I_n$, we have
\begin{equation}\label{eq:Qqn_height_infinite_limit_exponents_bound}
a(s_n \bm{k}) \gg \lp s_n^{\frac{r+1}{2}} \ell_{x,n} \rp^r .
\end{equation}
Since $\a_n\to\infty$ as $n \to \infty$, we have
\begin{equation}\label{eq:Qqn_height_infinite_limit_exponents}
\lim_{n \to \infty} \!\lp s_n^{\frac{r+1}{2}} \ell_{x,n} \rp
= \infty.
\end{equation}
So the set $I_n$ satisfies Corollary \ref{C:EquivalenceofEnsembles} (4) and it is enough to prove the claimed limit for $Q_{q_n}$ as\footnote{
In fact, as in Theorem \ref{T:lgstajk} one can prove a stronger result saying that
\begin{equation*}
-\log\!\left(Q_{q_n}(H\leq \ell_{x_n,n})\right) = e^{-x_n}(1+o(1))+o(1)
\mbox{ as } n \to \infty
\end{equation*}
for any sequence of real numbers $x_n$ such that
$- \a_n^{\frac{1}{4}} \leq x_n \leq \a_n^{\frac{1}{4}}$
for all $n$.
}
\begin{equation}\label{eq:Qqn_height_distribution_asymptotic_goal}
\lim_{n \to \infty} Q_{q_n} \!\left(H \leq\ell_{x,n}\right)=e^{-e^{-x}} .
\end{equation}

We start our discussion with
\begin{equation*}
- \log \!\lp Q_{q_n} (H \leq \ell_{x,n})  \rp
= - \sum_{\substack{\bm{k} \in \IN^r \\ L(\bm{k}-\bm{1}) > \ell_{x,n}}}
\log\!\lp 1 - q_n^{a (\bm{k})} \rp .
\end{equation*}
Thanks to the bound \eqref{eq:Qqn_height_infinite_limit_exponents_bound} for the terms contributing to this sum and the limit \eqref{eq:Qqn_height_infinite_limit_exponents},
we can assume that $n$ is sufficiently large such that $q_n^{a (\bm{k})} \in [0,\frac{1}{2}]$ for all the terms in the sum.
In particular, recalling the inequality $t \leq - \log (1-t) \leq t+t^2$ for $t \in [0,\frac{1}{2}]$ as remarked in the proof of Theorem \ref{T:lgstajk}
and separating the contributions from $\LL_1$, $\LL_2$, and $\LL$ defined in \eqref{eq:Qqn_height_distribution_Lambda_set_definition}, we find
\begin{equation*}
- \log \!\lp Q_{q_n} (H \leq \ell_{x,n})  \rp
=
\sum_{\substack{\bm{k} \in \LL_1 \\ L(\bm{k}-\bm{1}) > \ell_{x,n}}} \!\!\!\!\!\!
q_n^{a (\bm{k})}
+ O\!\lp
\sum_{\substack{\bm{k} \in \LL_1 \\ L(\bm{k}-\bm{1}) > \ell_{x,n}}}  \!\!\!\!\!\!
q_n^{2a (\bm{k})}
+
\sum_{\substack{\bm{k} \in \LL_2 \\ L(\bm{k}-\bm{1}) > \ell_{x,n}}}  \!\!\!\!\!\!
q_n^{a (\bm{k})}
+
\sum_{\substack{\bm{k} \in \LL \\ L(\bm{k}-\bm{1}) > \ell_{x,n}}}  \!\!\!\!\!\!
q_n^{a (\bm{k})}
\rp .
\end{equation*}
Here we note that for $\bm{k} \in \LL_1$, i.e., vectors of the form $\bm{k} = (k,1,\ldots,1)$ or $\bm{k} = (1,\ldots,1,k)$, we have
\begin{equation*}
a(\bm{k}) = \frac{k(k+1)\cdots (k+r-1)}{r!}
\andd
L(\bm{k}-\bm{1}) = \frac{r}{2} (k-1) .
\end{equation*}
So we have
\begin{equation*}
\sum_{\substack{\bm{k} \in \LL_1 \\ L(\bm{k}-\bm{1}) > \ell_{x,n}}}
q_n^{a (\bm{k})}
=
2 \!\! \sum_{\substack{k \in \IN \\ k > \frac{2}{r} \ell_{x,n} + 1}} \!\!\!\!
q_n^{\frac{1}{r!} k(k+1)\cdots (k+r-1)}
\mbox{ and }
\sum_{\substack{\bm{k} \in \LL_1 \\ L(\bm{k}-\bm{1}) > \ell_{x,n}}}
q_n^{2a (\bm{k})} \leq
2 \sum_{\substack{k \in \IN \\ k > \frac{2}{r} \ell_{x,n} + 1}}
q_n^{\frac{2}{r!} k^r} .
\end{equation*}
Moreover, by Lemma \ref{lem:height_representation_lambda2_bounds} (and using that $r \geq 2$ to further bound $a(\bm{k}) \geq \frac{2}{r!} k^r$) we find
\begin{equation*}
\sum_{\substack{\bm{k} \in \LL_2 \\ L(\bm{k}-\bm{1}) > \ell_{x,n}}}
q_n^{a (\bm{k})}
\leq
2 (r-1) \sum_{\substack{k \in \IN \\ k > \frac{2}{r} \ell_{x,n} - r}}
q_n^{\frac{2}{r!} k^r} .
\end{equation*}
So these facts lead to
\begin{equation}\label{eq:Qqn_height_distribution_three_terms}
- \log \!\lp Q_{q_n} (H \leq \ell_{x,n})  \rp
=
2 \!\! \sum_{\substack{k \in \IN \\ k > \frac{2}{r} \ell_{x,n} + 1}} \!\!\!\!\!\!
q_n^{\frac{1}{r!} k(k+1)\cdots (k+r-1)}
+
O \!\lp \sum_{\substack{k \in \IN \\ k > \frac{2}{r} \ell_{x,n} - r}}
q_n^{\frac{2}{r!} k^r}
+\!\!\!\!
\sum_{\substack{\bm{k} \in \LL \\ L(\bm{k}-\bm{1}) > \ell_{x,n}}} \!\!\!\!\!\!
q_n^{a (\bm{k})} \rp .
\end{equation}

We next study \eqref{eq:Qqn_height_distribution_three_terms} and start with the main term.
Using $k^r \leq k (k+1) \cdots (k+r-1) \leq (k+r)^r$ and the fact that
for $k > \frac{2}{r} \ell_{x,n}$ each summand is $o(1)$ by \eqref{eq:Qqn_height_infinite_limit_exponents}, we obtain
\begin{equation}\label{eq:Qqn_height_distribution_main_term_dropped_summands}
\sum_{\substack{k \in \IN \\ k > \frac{2}{r} \ell_{x,n} + 1}} \!\!\!\!
q_n^{\frac{1}{r!} k(k+1)\cdots (k+r-1)}
=
\sum_{\substack{k \in \IN \\ k > \frac{2}{r} \ell_{x,n}}} \!\!\!\!
q_n^{\frac{1}{r!} k^r}
+ o(1).
\end{equation}
Now using Euler's summation formula on this remaining sum, employing \eqref{eq:Qqn_height_infinite_limit_exponents} to bound the error terms, and recalling the definition of incomplete gamma functions in \eqref{eq:incomplete_gamma_definition} we get
\begin{equation*}
\sum_{\substack{k \in \IN \\ k > \frac{2}{r} \ell_{x,n} + 1}} \!\!\!\!
q_n^{\frac{1}{r!} k(k+1)\cdots (k+r-1)}
=
\frac{r!^{\frac{1}{r}}}{r} s_n^{-\frac{r+1}{2}}
\GG\!\lp \frac{1}{r}, \frac{1}{r!} \lp \frac{2}{r} \ell_{x,n} s_n^{\frac{r+1}{2}}\rp^r \rp
+o(1) .
\end{equation*}
Using the asymptotic behavior of incomplete gamma functions from \eqref{eq:incomplete_gamma_asymptotic_behavior} then yields
\begin{equation*}
2\!\!\!\!\sum_{\substack{k \in \IN \\ k > \frac{2}{r} \ell_{x,n} + 1}} \!\!\!\!
q_n^{\frac{1}{r!} k(k+1)\cdots (k+r-1)}
=
2\GG(r) s_n^{-\frac{r+1}{2}} \lp \frac{2}{r} \ell_{x,n} s_n^{\frac{r+1}{2}}\rp^{-(r-1)}
e^{-\frac{1}{r!} \lp \frac{2}{r} \ell_{x,n} s_n^{\frac{r+1}{2}}\rp^r}
(1+o(1))
+o(1) .
\end{equation*}
Recalling the definition of
$\a_n$ from \eqref{eq:introduction_C_r_alpha_n_definitions}
and of $\ell_{x,n}$ from \eqref{eq:Qqn_height_ell_x_n_definition}, we have
\begin{equation}\label{eq:Qqn_height_I_x_n_definition}
\CI_n :=
- \log \!\lp
2\GG(r) s_n^{-\frac{r+1}{2}} \lp \frac{2}{r} \ell_{x,n} s_n^{\frac{r+1}{2}}\rp^{-(r-1)}
e^{-\frac{1}{r!} \lp \frac{2}{r} \ell_{x,n} s_n^{\frac{r+1}{2}}\rp^r} \rp
=
x+o(1) .
\end{equation}
So the main term in \eqref{eq:Qqn_height_distribution_three_terms} tends to $e^{-x}$ as $n \to \infty$, and to prove \eqref{eq:Qqn_height_distribution_asymptotic_goal} we only need to show that the remaining two error terms are both $o(1)$.

We start with the easier single variable sum. As in \eqref{eq:Qqn_height_distribution_main_term_dropped_summands}, the summands are $o(1)$ as $n \to \infty$, so
\begin{equation*}
\sum_{\substack{k \in \IN \\ k > \frac{2}{r} \ell_{x,n} - r}}
q_n^{\frac{2}{r!} k^r}
=
\sum_{\substack{k \in \IN \\ k > \frac{2}{r} \ell_{x,n} +1}}
q_n^{\frac{2}{r!} k^r}
+o(1) .
\end{equation*}
Furthermore, the summand is monotonically decreasing in $k$, so we can put an upper bound on the remaining sum using its integral:
\begin{equation*}
\sum_{\substack{k \in \IN \\ k > \frac{2}{r} \ell_{x,n} + 1}}  q_n^{\frac{2}{r!} k^r}
\leq
\int_{\frac{2}{r} \ell_{x,n}}^\infty e^{-\frac{2}{r!} s_n^{\frac{1}{2} r(r+1)} t^r} dt
=
\frac{1}{r} \lp \frac{r!}{2} \rp^{\frac{1}{r}} s_n^{-\frac{r+1}{2}}
\GG\!\lp \frac{1}{r}, \frac{2}{r!} \lp \frac{2}{r} \ell_{x,n} s_n^{\frac{r+1}{2}}\rp^r \rp.
\end{equation*}
So using the asymptotic behavior of incomplete gamma functions in \eqref{eq:incomplete_gamma_asymptotic_behavior} once again and recalling the definition of $\CI_n$ in equation \eqref{eq:Qqn_height_I_x_n_definition}, we obtain
\begin{equation*}
\sum_{\substack{k \in \IN \\ k > \frac{2}{r} \ell_{x,n} + 1}}  q_n^{\frac{2}{r!} k^r}
\ll
s_n^{-\frac{r+1}{2}}
\lp \ell_{x,n} s_n^{\frac{r+1}{2}}\rp^{-(r-1)}
e^{-\frac{2}{r!} \lp \frac{2}{r} \ell_{x,n} s_n^{\frac{r+1}{2}}\rp^r}
\asymp
e^{-\CI_n - \frac{1}{r!} \lp \frac{2}{r} \ell_{x,n} s_n^{\frac{r+1}{2}}\rp^r}.
\end{equation*}
This is $o(1)$ as we want to prove thanks to \eqref{eq:Qqn_height_infinite_limit_exponents} and
\eqref{eq:Qqn_height_I_x_n_definition}.

Our final task in \eqref{eq:Qqn_height_distribution_three_terms} is the sum over $\LL$.
For this purpose, we employ Lemma \ref{lem:height_representation_dominant_weights} supposing that $n$ is sufficiently large so that $\ell_{x,n} > N_r$.
Then noting that $L(\bm{k}-\bm{1})$ assumes integral and half-integral values
and that the number of $\bm{k} \in \LL$ with $k = 2L(\bm{k}-\bm{1})$ is $\ll k^{r-1}$, we find
\begin{equation*}
\sum_{\substack{\bm{k} \in \LL \\ L(\bm{k}-\bm{1}) > \ell_{x,n}}} \!\!\!\!
q_n^{a (\bm{k})}
\ll
\sum_{\substack{k \in \IN \\ k > 2 \ell_{x,n}}}
k^{r-1} q_n^{\frac{3}{2(r-1)!}  \lp \frac{k}{r} \rp^r}  .
\end{equation*}
The summand on the right is monotonically decreasing for $s_n^{\frac{r+1}{2}} k \gg 1$.
So by \eqref{eq:Qqn_height_infinite_limit_exponents}, we can assume that $n$ is sufficiently large such that it is monotonic decreasing for $k \geq 2 \ell_{x,n} - 1$.
Then we bound
\begin{align*}
\sum_{\substack{k \in \IN \\ k > 2 \ell_{x,n}}} k^{r-1}
q_n^{\frac{3}{2(r-1)!} \lp \frac{k}{r} \rp^r}
&\leq
\int_{2 \ell_{x,n}-1}^\infty t^{r-1}
\exp\!\lp -\frac{3}{2}\frac{1}{(r-1)!} s_n^{\frac{r(r+1)}{2} } \lp \frac{t}{r} \rp^r \rp dt
\\
&\ll
s_n^{-\frac{r(r+1)}{2}}
\exp\!\lp - \frac{3}{2} \frac{1}{(r-1)!}
\lp \frac{2}{r} \ell_{x,n} s_n^{\frac{r+1}{2}} \rp^r \rp .
\end{align*}
According to the definition of $\a_n$, we have
\begin{equation*}
(2 \GG(r))^r s_n^{-\frac{r(r+1)}{2} } = \exp\!\lp \frac{\a_n}{(r-1)!} \rp,
\end{equation*}
so noting that
\begin{equation*}
\lp \frac{2}{r} \ell_{x,n} s_n^{\frac{r+1}{2}} \rp^r - \frac{2\a_n}{3}
=
\frac{\a_n}{3}  + \GG(r+1) \lp x -  \frac{r-1}{r} \log (\a_n)  \rp
+ o(1) .
\end{equation*}
tends to infinity (because $\a_n$ does so) proves that
this final term is also $o(1)$, which completes the proof of equation \eqref{eq:Qqn_height_distribution_asymptotic_goal}.
\end{proof}

\section{Limit shape and total number of irreducible representations}\label{S:LimitShapeTotalIrreps}

Our approaches to the limit shape and distribution of the total number of irreducible representations are similar in that, unlike the distributions in Sections \ref{S:SmallMultiplicities} and \ref{S:MaxDimHeight}, we apply the rare events lemma if we cannot readily use equivalence of ensembles.

\subsection{Limit shape}

For $r=1$, a representation $\rho$ corresponds to a partition and a natural visual representation of $\rho$ is its Young diagram. The shape of a Young diagram is described by the random variable $\varphi(t):=\sum_{k \geq t} X_k.$
For $r>1$, we define the analogous shape function
$$\varphi(\bm{t}):=\sum_{\substack{\k \in \N^r \\ k_j \geq t_j}} X_{\k}$$
and use the notation  $\varphi_{\rho}(\bm{t})$ to denote the value of the random variable $\varphi(t)$ at the representation~$\rho$.
To discuss the asymptotic distribution of $\varphi(\bm{t})$ under $P_n$, we first give a technical lemma that we use in approximating logarithms. The proof is elementary and is omitted.
\begin{lem}\label{L:logxqineq}
Let $q_0\in (0,1)$ and $\log(q_0)<x_0<0$.  Then there exists $C=C(x_0,q_0)$ such that
$$
0\leq \log\left(\frac{1-q}{1-e^{-x}q}\right)+\frac{xq}{1-q} \leq C\frac{x^2q}{(1-q)^2} \qquad \text{for $x\geq x_0$ and $q\in[0,q_0]$.}
$$
\end{lem}
\noindent We next show that $\varphi(\bm{t})$ converges in probability to $f_r(\bm{t})$ defined in \eqref{E:limitshapedef}, under appropriate scaling.

\begin{thm}\label{T:Shape}
For any $\varepsilon, \eta>0$, we have
$$
\lim_{n \to \infty} P_n\left(\sup_{\bm{t}\in[\eta,\infty)^r}
\left|s_n^r\varphi\!\left(s_n^{-1}\bm t\right)-f_r(\bm{t})\right| \leq \varepsilon\right) = 1.
$$
\end{thm}
\begin{proof}
Let $c \in \mathbb{N}_0$, and $\bm{\kappa} \in \N^r$ with $s_n\bm{\kappa}\in[\eta,\infty)^r$.
Note that $\varphi\!\left(\bm{\kappa}\right)$ is a function of $X_{\k}$ with $\k \in I_{\bm{\kappa},n}$, where 
$$
I_{\bm{\kappa},n}:=\left\{\k\in \mathbb{N}^r : k_j \geq \kappa_j \right\}.
$$
Since $I_{\bm{\kappa},n}$ depends on multiplicities with $a(s_n\bm{k}) \asymp 1$, we cannot employ equivalence of ensembles via Corollary \ref{C:EquivalenceofEnsembles}.
Instead, we use a saddle-point/Chernoff-type bound, together with Lemma \ref{lem:ExpSmallPrinciple}.

We start by noting that
\begin{align*}
\#\!\left\{\rho : \dim (\rho)=n, \ \varphi_{\rho}\!\left(\bm{\kappa}\right)=c \right\}
&=\coef{q^nz^c} \lp \prod_{\k \in I_{\bm{\kappa},n}}\frac{1}{1-zq^{a(\k)}}\prod_{\k \in \N^r \setminus I_{\bm{\kappa},n}}\frac{1}{1-q^{a(\k)}} \rp.
\end{align*}
Now focus on the values $q=q_n$ and $z=e^{-x}$ with $x>-a(s_n\bm{\kappa})$, where the latter ensures that
$$
zq_n^{a(\k)} = e^{-x-a(s_n\k)}  <1 \quad \text{for $\k \in I_{\bm{\kappa},n}$.}
$$
Plugging in these values with an extra factor of $q^{-n} z^{-c}$ we find the bound
\begin{align*}
\#\left\{\rho : \dim (\rho)=n, \ \varphi_{\rho}\!\left(\bm{\kappa}\right)=c\right\}
&\leq q_n^{-n}e^{cx} \prod_{\k \in I_{\bm{\kappa},n}}\frac{1}{1-e^{-x}q_n^{a(\k)}}\prod_{\k \in \N^r \setminus I_{\bm{\kappa},n}}\frac{1}{1-q_n^{a(\k)}},
\end{align*}
so that
$$
Q_{q_n}\!\left(\dim =n, \ \varphi\!\left(\bm{\kappa}\right)=c \right) \leq e^{cx} \prod_{\k \in I_{\bm{\kappa},n}}\frac{1-q_n^{a(\k)}}{1-e^{-x}q_n^{a(\k)}}.
$$
We now use Lemma \ref{L:logxqineq} with $q_0=e^{-a(\eta, \dots, \eta)}\geq e^{-a(s_n\bm{\kappa})}$ and $x_0 = - \frac{1}{2} a(\eta, \dots, \eta)$, and assume from now on that\footnote{Note that this ensures $x > - a(s_n \bm{\kappa})$ because $x_0 > -a(\eta, \dots, \eta) \geq - a(s_n \bm{\kappa})$.} $x\geq x_0$.  Note that for $\k \in I_{\bm{\kappa},n}$ we have $q_n^{a(\k)} \leq q_0$.  Thus,
\begin{equation*}
\log \!\left(Q_{q_n}\!\left(\dim=n, \ \varphi\!\left(\bm{\kappa}\right)=c \right)\right) \leq cx-x\sum_{\k \in I_{\bm{\kappa},n}}\frac{q_n^{a(\k)}}{1-q_n^{a(\k)}}+O\!\left(x^2\sum_{\k \in I_{\bm{\kappa},n}}\frac{q_n^{a(\k)}}{\left(1-q_n^{a(\k)}\right)^2}\right),
\end{equation*}
where the implied constant in the error above depends only on $\eta$.
The two sums above can be approximated with integrals over $\prod_{j=1}^r [s_n\kappa_j,\infty)$ as in the proof of Proposition \ref{P:SaddlePointAsympGrowth} to get
$$
\sum_{\k \in I_{\bm{\kappa},n}}\frac{q_n^{a(\k)}}{1-q_n^{a(\k)}}s_n^r= f_r(s_n\bm{\kappa}) + O(s_n), \quad \text{and} \quad
\sum_{\k \in I_{\bm{\kappa},n}}\frac{q_n^{a(\k)}}{\lp 1-q_n^{a(\k)} \rp^2} s_n^r =O(1),
$$
and hence
$$
\log \!\left(Q_{q_n}\!\left(\dim=n, \ \varphi\!\left(\bm{\kappa}\right)=c \right)\right)\leq s_n^{-r}\left(x\left(cs_n^r-f_r(s_n\bm{\kappa})\right)+O\!\left(|x|s_n+x^2\right)\right),
$$
where the error terms depend only on $\eta$. 
Now, if $|cs_n^r-f_r(s_n\bm{\kappa})| > \frac{\varepsilon}{2}$, then we can choose $x$ sufficiently close to zero and positive or negative as necessary, so that for some constant $C_{\varepsilon,\eta}>0$ and $n$ sufficiently large (depending on $\eta,\varepsilon$),
$$
Q_{q_n}\!\left(\dim=n, \ \varphi\!\left(\bm{\kappa}\right)=c \right) \leq 
e^{-C_{\varepsilon,\eta}s_n^{-r}}
\quad \mbox{if }  \left|cs_n^r-f_r(s_n\bm{\kappa})\right| > \frac{\varepsilon}{2}.
$$

Now note that  $f_r (\bm{t})$ is decreasing in each of its components and for $\bm{x} \in [0,1]^r$ we have
\begin{equation*}
0 \leq f_r (\bm{t}) - f_r (\bm{t} + s_n \bm{x}) \ll_\eta s_n .
\end{equation*}
So we can assume that $n$ is sufficiently large (depending on $\eta$) such that this difference is $\leq \frac{\varepsilon}{2}$. 
Since $\varphi (s_n^{-1}\bm{t}) = \varphi (\bm{k}_{n,\bm{t}})$ where 
$\bm{k}_{n ,\bm{t}} := \lceil s_n^{-1}\bm{t} \rceil$ (with the ceiling function applied to the components), we find that for such $n$ we have
\begin{multline*}
Q_{q_n}\!\left(\dim=n, \ \left|s_n^r \varphi\!\left(s_n^{-1}\bm t\right)-f_r(\bm{t})\right| > \varepsilon \ \text{for some} \ \bm{t}\in[\eta,\infty)^r \right) 
\\
\leq 
Q_{q_n}\!\left(\dim=n, \ \left|s_n^r \varphi(\bm{\kappa})-f_r(s_n \bm{\kappa})\right| > \frac{\varepsilon}{2} \ \text{for some} \ \bm{\kappa} \in \IN^r \mbox{ with } 
s_n \bm{\kappa}\in[\eta,\infty)^r \right)  .
\end{multline*}
Note that $\dim (\rho)=n$ implies $\varphi_\rho \!\left(\bm{\kappa}\right) \leq n$ and if $X_{\k}(\rho)>0$ that $k_j\leq n$
for $j\in\{1,\dots,r\}$.
In particular, $\varphi_\rho (\bm{\kappa}) = 0$ if $\kappa_j > n$ for some $j\in\{1,\dots,r\}$. Also, we can assume that $n$ is sufficiently large such that $f_r(s_n \bm{\kappa}) < \frac{\varepsilon}{2}$ for such $\bm{\kappa}$ because $s_n n \to \infty$ as $n \to \infty$. For such $n$ we then find
\begin{align*}
&Q_{q_n}\!\left(\dim=n, \ \left|s_n^r \varphi(\bm{\kappa})-f_r(s_n \bm{\kappa})\right| > \frac{\varepsilon}{2} \ \text{for some} \ \bm{\kappa} \in \IN^r \mbox{ with } 
s_n \bm{\kappa}\in[\eta,\infty)^r \right) 
\\
&\hspace{2cm}=
Q_{q_n}\!\left(\dim=n, \ \left|s_n^r \varphi(\bm{\kappa})-f_r(s_n \bm{\kappa})\right| > \frac{\varepsilon}{2} \ \text{for some} \ \bm{\kappa} \in \IN^r \mbox{ with } 
s_n^{-1} \eta \leq \kappa_j \leq n \right) 
\\
&\hspace{2cm}\leq 
\sum_{\substack{\bm{\kappa} \in \IN^r \\ s_n^{-1} \eta \leq \kappa_j \leq n}}\ 
\sum_{\substack{c \in \{0,1,\ldots,n\} \\ |s_n^r c - f_r(s_n \bm{\kappa})|>\frac{\varepsilon}{2}}} 
Q_{q_n}\!\left(\dim=n, \  \varphi(\bm{\kappa}) = c \right) . 
\end{align*}
Therefore, we obtain
\begin{equation*}
Q_{q_n}\!\left(\dim=n, \ \sup_{\bm{t} \in [\eta,\infty)^r} \left|s_n^r \varphi\!\left(s_n^{-1}\bm t\right)-f_r(\bm{t})\right| > \varepsilon  \right) \ll 
n^{r+1} e^{-C_{\varepsilon,\eta}s_n^{-r}} .
\end{equation*}
The theorem statement then follows by Lemma \ref{lem:ExpSmallPrinciple}.
\end{proof}

\subsection{Total number of irreducible representations}
We finally consider the total number of irreducible representations, which is counted by the random variable
$$
N:=\sum_{\k \in \N^r} X_{\k}.
$$
The distribution of $N$ depends on the following product.

\begin{lem}\label{L:MomentProductConv}
The product
\begin{equation*}
M(z):=\prod_{\k \in \N^r} \left(1- \frac{z}{a(\k)}\right)^{-1}
\end{equation*}
yields a meromorphic function on $\IC$ with poles located at the values $a(\bm{k})$ with $\bm{k} \in \IN^r$. 
For some constant $D>0$, it is bounded on the imaginary axis as
\begin{equation*}
|M(it)| \leq \exp \!\lp - D |t|^{\frac{2}{r+1}} \rp
\quad \mbox{if } |t| \geq 1 .
\end{equation*}
\end{lem}
\begin{proof}
We have $\smash{\sum_{\bm{k} \in \IN^r} a(\bm{k})^{-1} < \infty}$ thanks to abel partial summation and Proposition \ref{prop:irrep_partial_sum_asymptotics}. 
Thus the reciprocal product $M^{-1}(z)$ converges locally uniformly to an entire function with zeros on $\smash{\{a(\bm{k})\}_{\bm{k} \in \IN^r}}$, from which the first statement follows.
To prove the second statement, we write
\begin{equation*}
|M(it)| = 
\exp \!\lp -\frac{1}{2} \sum_{\k \in \N^r} \log\!\lp 1+\frac{t^2}{a(\k)^2} \rp \rp
=
\exp \!\lp -\frac{1}{2} \sum_{m \geq 1} \log\!\lp 1+\frac{t^2}{m^2} \rp \varrho_r (m) \rp 
\quad \mbox{for } t \in \IR.
\end{equation*}
An application of abel partial summation and Proposition \ref{prop:irrep_partial_sum_asymptotics} then proves the claimed estimate.
\end{proof}

The limiting cumulative distribution function for $N$ likely lacks a simple description in terms of elementary functions, so instead we state our result in terms of its moment generating function.

\begin{thm}\label{T:NumberofParts}
If governed by $P_n$, the random variable \smash{$s_n^{\frac{r(r+1)}{2}}N$} converges in distribution as $n \to \infty$ to the unique real random variable that has $M(u)$ as its moment generating function for $-1 < u < 1$.
\end{thm}

Our approach to Theorem \ref{T:NumberofParts} is to break up $N$ into sums over small irreducible representations and large irreducible representations.  Equivalence of ensembles applies to the small irreducible representations.  The contribution of large irreducible representations, on the other hand, turns out to be negligible, which we prove using the rare events lemma.

\begin{lem}\label{L:NumPartsTailSmall}
Let $\l_r := \frac{r(r+1)}{24}$, $r \geq 2$, and
$I_n:=\left\{\k \in \N^r : a(\k) > s_n^{-5 \l_r}\right\}$.
For all $\varepsilon>0$, we have
$$\lim_{n \to \infty}  P_n\!\left(s_n^{12 \l_r} \sum_{\k \in I_n} X_{\k} \leq \varepsilon\right)=1.$$
\end{lem}
\begin{proof}
Let $\ell \in \{1,2,\ldots,n\}$ with $s_n^{12 \l_r} \ell>\varepsilon$ (note that $s_n^{12 \l_r} n \asymp n^{\frac{2}{r+3}}$ by Proposition \ref{P:SaddlePointAsympGrowth}).
Using a saddle-point/Chernoff bound, as in the proof of Theorem \ref{T:Shape}, gives
$$
Q_{q_n}\!\left(\dim = n, \ \sum_{\k \in I_n}X_{\k}=\ell\right)
\leq e^{-\ell x_n}\prod_{\k \in I_n} \frac{1-q_n^{a(\k)}}{1-e^{x_n}q_n^{a(\k)}},
$$
where $x_n := s_n^{8 \l_r}$, noting that $e^{x_n}q_n^{a(\k)}<1$ for $\k \in I_n$ and $n$ sufficiently large.
By the mean value theorem, we can bound
$$
\left|-\log\left(1-e^{x_n}q_n^{a(\k)}\right)+\log\left(1-q_n^{a(\k)}\right)\right|\leq \frac{x_ne^{x_n}q_n^{a(\k)}}{1-e^{x_n}q_n^{a(\k)}} \leq \frac{2x_nq_n^{a(\k)}}{1-e^{x_n}q_n^{a(\k)}},
$$
for $n$ sufficiently large, so
\begin{equation}\label{E:ZnMVTbound}
\log \!\left(Q_{q_n}\!\left(\dim = n, \ \sum_{\k \in I_n} X_{\k}=\ell\right)\right) \leq -x_n\left(\ell - \sum_{k \in I_n} \frac{2q_n^{a(\k)}}{1-e^{x_n}q_n^{a(\k)}}\right) .
\end{equation}
We next show that the sum above is $o(s_n^{-12\l_r})$.  Using abel partial summation we calculate
\begin{equation*}
\sum_{k \in I_n} \frac{q_n^{a(\k)}}{1-e^{x_n}q_n^{a(\k)}} 
=
-\frac{e^{-s_n^{7\l_r}}}{1-e^{s_n^{8\l_r}-s_n^{7\l_r}}}R_r\!\left(s_n^{-5 \l_r}\right)+s_n^{12\l_r} \int_{s_n^{-5\l_r}}^{\infty} \frac{e^{-s_n^{12\l_r}t} \, R_r(t)}{\left(1-e^{s_n^{8\l_r}-s_n^{12\l_r}t}\right)^2}dt .
\end{equation*}
Thanks to  Proposition \ref{prop:irrep_partial_sum_asymptotics}, the boundary term is
$\ll s_n^{-7 \l_r} ( s_n^{-5 \l_r } )^{\frac{2}{r+1}} = o(s_n^{-12\l_r})$ as $n \to \infty$. For the integral, we first study the portion from $s_n^{-5 \l_r}$ to $s_n^{-12 \l_r}$ and find that
\begin{equation*}
s_n^{12\l_r} \int_{s_n^{-5\l_r}}^{s_n^{-12\l_r}} \frac{e^{-s_n^{12\l_r}t} \, R_r(t)}{\left(1-e^{s_n^{8\l_r}-s_n^{12\l_r}t}\right)^2}dt
\ll
s_n^{-12\l_r} \int_{s_n^{-5\l_r}}^{s_n^{-12\l_r}}  t^{-2+\frac{2}{r+1}} dt
\asymp 
s_n^{-12\l_r} \lp s_n^{-5\l_r} \rp^{-\frac{r-1}{r+1}},
\end{equation*}
which is $o(s_n^{-12\l_r})$ as well. The remaining portion is estimated as
\begin{equation*}
s_n^{12\l_r} \int_{s_n^{-12\l_r}}^{\infty} \frac{e^{-s_n^{12\l_r}t} \, R_r(t)}{\left(1-e^{s_n^{8\l_r}-s_n^{12\l_r}t}\right)^2}dt
\ll
s_n^{12\l_r} \int_{s_n^{-12\l_r}}^{\infty} e^{-s_n^{12\l_r}t} \, t^{\frac{2}{r+1}} dt
\asymp
\lp s_n^{-12\l_r} \rp^{\frac{2}{r+1}}
\end{equation*}
and this is also $o(s_n^{-12\l_r})$.
Plugging this into \eqref{E:ZnMVTbound} and recalling $s_n^{12 \l_r} \ell>\varepsilon$ and $x_n = s_n^{8 \l_r}$, we~find
\begin{equation*}
\log \!\left(Q_{q_n}\!\left(\dim = n, \ \sum_{\k \in I_n} X_{\k}=\ell\right)\right) \leq -s_n^{-4\l_r} \lp \varepsilon + o(1) \rp.
\end{equation*}
Since the error term does not depend on $\ell$, running over all possible $\ell$, we have the bound
$$
Q_{q_n}\!\left(\dim = n, \ s_n^{12\l_r}\sum_{\k \in I_n} X_{\k} > \varepsilon  \right)
\leq n \exp \!\lp  -\varepsilon s_n^{-4\l_r} (1+o(1)) \rp.
$$
As the right-hand side decays exponentially, the lemma follows using Lemma \ref{lem:ExpSmallPrinciple}.
\end{proof}

We are now ready to prove Theorem \ref{T:NumberofParts}.
\csname @beginparpenalty\endcsname10000
\begin{proof}[Proof of Theorem \ref{T:NumberofParts}]
We assume the notations introduced in Lemma \ref{L:NumPartsTailSmall}. Our first task is establishing that the moment generating function for 
$$
	s_n^{\frac{r(r+1)}{2}}\sum_{\k \in \N^r\setminus I_n}X_{\k}
$$
if governed by $Q_{q_n}$ tends pointwise to $M(u)$ for $|u| < 1$. For this, we obtain, using Proposition \ref{P:BoltzmannSimpleProperties}, that, for any fixed $|u|<1$,
\begin{equation*}
\mathrm{E}_{q_n}\!\left(\exp\left(u \, s_n^{\frac{r(r+1)}{2}}\!\!\sum_{\k \in \N^r\setminus I_n}\!\!X_{\k} \right)\right)
=
\prod_{\k \in \N^r\setminus I_n}\mathrm{E}_{q_n}\!\left(q_n^{-u X_{\k}} \right)
=
\prod_{\k \in \N^r\setminus I_n}\frac{1-q_n^{a(\k)}}{1-q_n^{a(\k)-u}}.
\end{equation*}
Since $s_n^{12 \l_r} \leq a(s_n \bm{k}) \leq s_n^{7 \l_r}$ for $\bm{k} \in \N^r\setminus I_n$, we find
\begin{equation*}
1-q_n^{a(\k)} = s_n^{12 \l_r} a (\bm{k}) \lp 1 + O\left(s_n^{7 \l_r}\right) \rp
\andd
1-q_n^{a(\k)-u} = s_n^{12 \l_r} (a (\bm{k}) -  u)  \lp 1 + O\left(s_n^{7 \l_r}\right) \rp
\end{equation*}
with the bounds on the error terms independent of $\bm{k}$. So taking the product over $\bm{k}$ while noting that $\# (\IN^r \setminus I_n) = R_r (s_n^{-5 \l_r}) \ll s_n^{-5 \l_r}$ by Proposition \ref{prop:irrep_partial_sum_asymptotics}, we obtain
\begin{equation*}
\mathrm{E}_{q_n}\!\left(\exp\left(u \, s_n^{\frac{r(r+1)}{2}}\!\!\sum_{\k \in \N^r\setminus I_n}\!\!X_{\k} \right)\right)
= \lp 1 + O\left(s_n^{2 \l_r}\right)\rp  \prod_{\k \in \N^r\setminus I_n} \!\!\lp 1- \frac{u}{a(\bm{k})} \rp^{-1}.
\end{equation*}
This indeed tends to $M(u)$ as we claimed.

Thanks to L\'evy's continuity theorem (adapted to moment generating functions as in \cite[Theorems 1, 3]{Cur}), this implies the existence of the unique random variable with moment generating function $M(u)$ for $-1 < u < 1$ and the weak convergence of the above mentioned sequence to this random variable. In particular, denoting the cumulative distribution function of this limiting random variable by $F$, we note that it is continuous (and in fact smooth) over $\IR$ since the characteristic function $M(it)$ satisfies $|M(it)| \ll |t|^{-K}$ as $|t| \to \infty$ for all $K \in \IN$ by Lemma \ref{L:MomentProductConv}.
Consequently, we have\footnote{Note that $F(x) = 0$ for $x < 0$ is alternatively seen from the property that the poles of $M(iz)$ are all located on the lower half-plane.}
\begin{equation}\label{eq:Qqn_total_number_In_moment_cdf}
\lim_{n \to \infty} Q_{q_n} \!\lp s_n^{\frac{r(r+1)}{2}}\sum_{\k \in \N^r\setminus I_n}X_{\k} \leq x \rp = F(x) 
\quad \mbox{ for all } x \in \IR.
\end{equation}

Since $N \geq \sum_{\k \in \N^r\setminus I_n}X_{\k}$, for any $x \in \IR$
we have the upper bound 
\begin{equation*}
P_n\!\left(s_n^{\frac{r(r+1)}{2}}N \leq x\right) 
\!\leq\! 
P_n\!\left(s_n^{\frac{r(r+1)}{2}}\sum_{\k \in \N^r\setminus I_n}X_{\k} \leq x\right) .
\end{equation*}
For $\bm{k} \in \N^r\setminus I_n$ we have $a(s_n \k) \leq s_n^{7 \l_r} = o(1)$, so we can use the equivalence of ensembles as in Corollary \ref{C:EquivalenceofEnsembles} (3) together with equation \eqref{eq:Qqn_total_number_In_moment_cdf} to find that 
\begin{equation*}
\limsup_{n \to \infty} P_n\!\left(s_n^{\frac{r(r+1)}{2}}N \leq x\right)
\leq F(x) .
\end{equation*}

Now let $\varepsilon >0$ be arbitrary. Then for any $x \in \IR$ we have the lower bound
\begin{equation*}
P_n\!\left(s_n^{\frac{r(r+1)}{2}}N \leq x\right)
\geq 
P_n\!\left(s_n^{\frac{r(r+1)}{2}}\sum_{\k \in \N^r\setminus I_n}X_{\k} \leq x-\varepsilon \right) 
+
P_n\!\left(s_n^{\frac{r(r+1)}{2}}\sum_{\k \in I_n}X_{\k} \leq \varepsilon \right) 
-1 .
\end{equation*}
The first term on the right-hand side converges to $F(x-\varepsilon)$ using the equivalence of ensembles as above, whereas the second and the third terms tend to zero thanks to Lemma \ref{L:NumPartsTailSmall}. So we have
\begin{equation*}
\liminf_{n \to \infty} P_n\!\left(s_n^{\frac{r(r+1)}{2}}N \leq x\right)
\geq F(x-\varepsilon) .
\end{equation*}
Since this bound holds for arbitrary $\varepsilon > 0$, we can use the continuity of the cumulative distribution function $F(x)$ to let $\varepsilon \to 0^+$ and obtain the same lower bound with $F$ as well.
\end{proof}

\section{Concluding remarks}\label{S:Conclusion}

A great deal more is known about partition statistics than the results stated in Table \ref{Table:partnstats}, and hence our results here for $\mathfrak{sl}_{r+1}(\mathbb{C})$-representations can likely be generalized and strengthened on several fronts.
\begin{enumerate}[leftmargin=*]
	\item In the Lie algebraic setting, it may be interesting to investigate families other than $\mathfrak{sl}_{r+1}(\mathbb{C})$, explore distributions under other measures, or incorporate infinite-dimensional Verma modules in a sensible way.
	\item Another direction is to extend our proof of equivalence of ensembles in the case of $\mathfrak{sl}_{r+1}(\mathbb{C})$-representations, which relied on using Weyl differencing and induction on $r$.  It would have wide applications to provide a criterion for equivalence of ensembles if the $a_r(\k)$ are replaced by generic multivariable polynomials.  This would apply to representations of all complex semi-simple Lie algebras.  We remark that, in the single variable polynomial case, the number of parts was studied by Goh and Hitczenko \cite{GH}.
	\item Finally, we consider only a selection of random variables on $\mathfrak{sl}_{r+1}(\mathbb{C})$-representations that seemed most natural.  We also chose to focus on convergence in distribution; however, it may be possible to adapt our techniques to prove refined local limit theorems such as convergence in terms of the L\'{e}vy--Prokhorov distance \cite{Fristedt}.  One could also undertake a more careful study of ``intermediate sized'' irreducible representations (those to which equivalence of ensembles does not apply) as in \cite{Pittel}.  For example, deviations from the limit shape could be studied as in \cite{DemboVershikZeitouni,Pittel}.  Random generation using the Boltzmann model could also be explored \cite{AD,BFR,DFLS}.
\end{enumerate}
 
\appendix \section{Irreducible Representation Dimensions of $\mathfrak{sl}_{3}(\mathbb{C})$ on a Circle}\label{A:LowerBounds}

In this appendix, we prove the statement of Proposition \ref{prop:sinsquare_sum_lower_bound_general_r_case} for the case $\mathfrak{sl}_3 (\IC)$, which forms the base case of the inductive argument used there.
We start with a number of technical lemmas.

\begin{lem}\label{lem:a2_lower_bound_successive_terms}
Let $N\in\N$, $0 < \ve < \frac{1}{5}$, and
$\frac{\ve}{N} \leq \th \leq \frac{1}{2} - \frac{\ve}{N}$.
Defining $b_k := \{ (2k+1) \th \}$ for $k \in \IZ$ and given
$\ell$ successive terms $b_{k_0+1}, b_{k_0+2}, \ldots, b_{k_0+\ell}$
in $[0,\frac{\ve}{2}]\cup[1-\frac{\ve}{2},1)$, the following hold:
\begin{enumerate}[label=\rm(\arabic*),leftmargin=*]
\item We have $\ell \leq \frac{N}{2} + 1$.
\item If $b_{k_0+\ell+1} \in (\frac{\ve}{2}, 1-\frac{\ve}{2})$, then we have
$b_{k_0+\ell+1}, b_{k_0+\ell+2}, \ldots, b_{k_0+2\ell} \in \lp \frac{\ve}{2}, 1-\frac{\ve}{2} \rp$.
\end{enumerate}
\end{lem}
\begin{proof}
First we assume $\frac{\ve}{N} \leq \th \leq \frac{1}{4}$ and
for any $k \in \IZ$ define
\begin{equation*}
n_k := \left\lfloor (2k+1)\th + \frac{1}{2} \right\rfloor \in \IZ
\andd
r_k := (2k+1)\th - n_k \in \lb - \frac{1}{2}, \frac{1}{2} \rp
\end{equation*}
so that $r_k \equiv b_k \Pmod{1}$ and $r_{k_0+1}, \ldots, r_{k_0+\ell} \in \lsb - \frac{\ve}{2}, \frac{\ve}{2} \rsb$.
Note that for $j \in \{1,2,\ldots,\ell-1 \}$ we have
\begin{equation*}
|n_{k_0+j+1} - n_{k_0+j}| \leq 2 \th + |r_{k_0+j+1} - r_{k_0+j}|
\leq \frac{1}{2} + \ve .
\end{equation*}
Therefore, $n_{k_0+j} = n_{k_0+1}$ for all $j \in \{1,2,\ldots,\ell \}$ and we find
\begin{equation}\label{eq:a2_lower_bound_successive_terms_bound}
2 \th (\ell - 1) =
r_{k_0+\ell} - r_{k_0+1} \leq \ve .
\end{equation}
Then $\th \geq \frac{\ve}{N}$ implies $\ell \leq \frac{N}{2} + 1$ and proves the first part.
The second claim, on the other hand, is trivial for $\ell = 1$, so assume $\ell \geq 2$.
For $j \in \{1,2,\ldots,\ell \}$ we have
\begin{equation*}
- \frac{\ve}{2} < r_{k_0+\ell} + 2 \th j \leq
\frac{\ve}{2} + 2 \th \ell \leq  \frac{5 \ve}{2}
\end{equation*}
with the last inequality following from \eqref{eq:a2_lower_bound_successive_terms_bound}.
Since $\ve < \frac{1}{5}$, we then find $r_{k_0+\ell} + 2 \th j \in \lp - \frac{\ve}{2}, \frac{1}{2} \rp$ for these values of $j$.
Now note that $r_{k_0+\ell} + 2 \th j \equiv r_{k_0+\ell+j} \pmod{1}$ and that we are assuming $r_{k_0+\ell+1}$ is in $[ -\frac{1}{2}, - \frac{\ve}{2} )
\cup ( \frac{\ve}{2}, \frac{1}{2} )$.
Therefore, $r_{k_0+\ell} + 2 \th = r_{k_0+\ell+1} \in \lsp \frac{\ve}{2}, \frac{1}{2} \rsp$ and consequently
$r_{k_0+\ell} + 2 \th j \in \lsp \frac{\ve}{2}, \frac{1}{2} \rsp$ for all $j \in \{1,2,\ldots,\ell \}$
because $r_{k_0+\ell} + 2 \th j$ is an increasing sequence in $j$, thereby proving the second part. The proof of both parts for the range $\frac{1}{4} \leq \th \leq \frac{1}{2} - \frac{\ve}{N}$ is similar.
\end{proof}

This lemma gives us our first step in our ladder of results towards
Proposition \ref{prop:sinsquare_sum_lower_bound_general_r_case} for $\mathfrak{sl}_3 (\IC)$.

\begin{lem}\label{lem:a2_lower_bound_window}
Let $N\in\N_{\geq 4}$, $0 < \ve \leq \frac{1}{32}$, and $k_0 \in \IZ$ with $3N \leq k_0 \leq 5N$. Then given any $\th \in \IR$ with
$\frac{\ve}{N} \leq \th \leq \frac{1}{2}$ we have
\begin{equation*}
\# \left\{ k \in \IZ : k_0 \leq k < k_0 + N \ \mathrm{and} \
\frac{\ve}{2} < \{(2k+1) \th \} < 1 - \frac{\ve}{2} \right\} \geq \frac{N}{8} .
\end{equation*}
\end{lem}
\begin{proof}
First assume $\th \in (\frac{1}{2} - \frac{\ve}{N}, \frac{1}{2}]$.
Then for $k_0 \leq k < k_0+N$, we have (using $0 < \ve \leq \frac{1}{32}$)
\begin{equation*}
(2k+1) \th - k = \frac{1}{2} - (2k+1) \lp \frac{1}{2} - \th \rp
\in \lp \frac{\ve}{2}, \frac{1}{2} \rb .
\end{equation*}
So all $N$ terms for $\{(2k+1) \th \}$ are in $\lp \frac{\ve}{2}, 1- \frac{\ve}{2} \rp$ and the lemma statement is satisfied.

Next we assume $\frac{\ve}{N} \leq \th \leq \frac{1}{2} - \frac{\ve}{N}$.
Defining $b_k := \{ (2k+1) \th \}$ for any $k \in \IZ$, consider the sequence $b_{k_0}, b_{k_0+1}, \ldots, b_{k_0+N-1}$. In that sequence let the first $\ell_1$ terms be in $\CI := \lsb 0, \frac{\ve}{2} \rsb \cup \lsb 1 - \frac{\ve}{2}, 1 \rsp$, the next $s_1$ terms be in $\CA:=(\frac{\ve}{2}, 1-\frac{\ve}{2})$, the next $\ell_2$ terms in $\CI$, $\ldots$ etc, and use these to
partition $N$ as
\begin{equation*}
(\ell_1 + s_1) + (\ell_2 + s_2) + \ldots + (\ell_m + s_m) = N .
\end{equation*}
Here $\ell_1, \ldots, \ell_m, s_1, \ldots, s_m$ ($m \in \IN$) are nonnegative integers where only $\ell_1$ and $s_m$ are allowed to be zero.
By Lemma \ref{lem:a2_lower_bound_successive_terms}, we have $s_j \geq \ell_j$ for all $j \in \{1,2,\ldots, m-1\}$ and that is including the possibility $\ell_1 = 0$ (where the inequality $s_1 \geq \ell_1$ is trivial).
Thus we get
\begin{align*}
&\hspace{-1cm}\# \left\{ k \in \IZ : k_0 \leq k < k_0+N \ \mathrm{and} \
\frac{\ve}{2} < \{(2k+1) \th \} < 1 - \frac{\ve}{2} \right\}
=
s_1 + s_2 + \ldots + s_m
\\
&\geq \frac{1}{2} \lp (\ell_1 + s_1) + (\ell_2 + s_2) + \ldots + (\ell_{m-1} + s_{m-1}) + s_m \rp
=
\frac{N-\ell_m}{2},
\end{align*}
which then implies the lemma statement for $N \geq 4$
because $\ell_m \leq \frac{N}{2} + 1$ by Lemma \ref{lem:a2_lower_bound_successive_terms}.
\end{proof}

As an immediate corollary we obtain the following result.

\begin{lem}\label{lem:a2_lower_bound_window_ladder1}
Let $0 < \ve \leq \frac{1}{32}$. For any $N\in\IN_{\geq 4}$ and $\th \in \IR$ with
$\frac{\ve}{N} \leq \th \leq \frac{1}{2}$ we have
\begin{equation*}
\# \left\{ k,j \in \IZ : \frac{3N}{2} \leq k < \frac{5N}{2}, \  2N \leq j \leq 3N,
\  \mathrm{and} \
\frac{\ve}{2} < \{(2k+2j+1) \th \} < 1 - \frac{\ve}{2} \right\} \geq \frac{N^2}{8} .
\end{equation*}
\end{lem}

We next lift this result to a quadratic polynomial via a Weyl differencing type argument.
\begin{lem}\label{lem:a2_lower_bound_window_ladder2}
Let $0 < \ve \leq \frac{1}{32}$. For any $N\in\IN_{\geq 4}$ and $\th \in \IR$ with
$\frac{\ve}{N^2} \leq \th \leq \frac{1}{2}$ we have
\begin{equation*}
\# \left\{ k,j \in \IZ : N \leq k \leq 3N, \ 2N \leq j \leq 3N, \  \mathrm{and} \
\frac{\ve}{4} < \{\th k (2j+k) \} < 1 - \frac{\ve}{4} \right\} \geq \frac{N^2}{16} .
\end{equation*}
\end{lem}
\begin{proof}
We prove the lemma for $\frac\ve N\leq\th\leq\frac12$ and
$\frac{\ve}{N(2N_h+1)}\leq\th\leq\frac{1}{2(2N_h+1)}$
with $N_h:=\lfloor\tfrac N2\rfloor$,
which cover the range of $\th$ assumed in the lemma statement.
So first assume that $\frac{\ve}{N} \leq \th \leq \frac{1}{2}$ and
define $b_{k,j} := k (2j+k)$.
If $\frac{\ve}{2} < \{ \th (b_{k+1,j}-b_{k,j}) \} < 1 - \frac{\ve}{2}$,
then we can not have both of $\{\th b_{k,j} \}$ and $\{\th b_{k+1,j} \}$
in the range $\lb 0, \frac{\ve}{4} \rb \cup \lb 1-\frac{\ve}{4}, 1 \rp$.
Thus we can place a lower bound
\begin{align*}
&\# \left\{ k,j \in \IZ : N \leq k \leq 3N, \ 2N \leq j \leq 3N, \  \mathrm{and} \
\frac{\ve}{4} < \{\th b_{k,j} \} < 1 - \frac{\ve}{4} \right\}
\\
&\quad\geq \frac{1}{2}
\# \left\{ k,j \in \IZ : N \leq k \leq 3N-1, \ 2N \leq j \leq 3N, \  \mathrm{and} \
\frac{\ve}{2} < \{\th  (b_{k+1,j}-b_{k,j}) \} < 1 - \frac{\ve}{2} \right\} .
\end{align*}
Since $b_{k+1,j}-b_{k,j} = 2k+2j+1$, we then use Lemma \ref{lem:a2_lower_bound_window_ladder1} to lower bound the second line by $\tfrac{N^2}{16}$.
The proof for the other range $\frac{\ve}{N(2N_h+1)} \leq \th \leq \frac{1}{2(2N_h+1)}$ follows the same argument while noting that $b_{k+N_h+1,j}-b_{k-N_h,j} = (2N_h+1) (2k+2j+1)$.
\end{proof}

Finally, we are ready to give such a lower bound for $a_2 (k,j) = \frac{1}{2} kj(k+j)$.
\begin{lem}\label{lem:a2_lower_bound_window_ladder3}
Let $0 < \ve \leq \frac{1}{32}$. For any $N\in\IN_{\geq 4}$ and $\th \in \IR$ with
$\frac{\ve}{N^3} \leq \th \leq \frac{1}{2}$ we have
\begin{equation*}
\# \left\{ k,j \in \IZ : N \leq k \leq 3N, \  N \leq j \leq 4N, \  \mathrm{and} \
\frac{\ve}{8} < \{\th a_2 (k,j) \} < 1 - \frac{\ve}{8} \right\} \geq \frac{N^2}{32} .
\end{equation*}
\end{lem}
\begin{proof}
The proof follows the same roadmap as Lemma \ref{lem:a2_lower_bound_window_ladder2}.
We prove the statement in the ranges $\frac{\ve}{N^2} \leq \th \leq \frac{1}{2}$ and
$\frac{\ve}{N^3} \leq \th \leq \frac{1}{2N}$, while noting that
$a_2 (k,j+1)-a_2 (k,j-1) = k(2j+k)$ and $a_2 (k,j+N)-a_2 (k,j-N) = N k(2j+k)$, respectively,
to lift the result of Lemma \ref{lem:a2_lower_bound_window_ladder2}.
\end{proof}

\end{document}